\documentclass[11pt]{amsart}
\RequirePackage{etex} 

\usepackage[vmargin=2.5cm,hmargin=2.5cm]{geometry}

\usepackage{amssymb, latexsym, amsfonts, amsmath, amsthm}

\newtheorem{theorem}{Theorem}[section]
\newtheorem{lemma}[theorem]{Lemma}
\newtheorem{corollary}[theorem]{Corollary}
\newtheorem{proposition}[theorem]{Proposition}

\newtheorem{definition}{Definition}[section]

\newtheorem{conjecture}[theorem]{Conjecture}
\newtheorem{question}[theorem]{Question}

\theoremstyle{remark}
\newtheorem{remark}[theorem]{Remark}
\newtheorem{example}[theorem]{Example}

\newcommand{\Z}{\mathbb{Z}}
\newcommand{\ZZ}{\operatorname{Z}}
\newcommand{\N}{\mathbb{N}}

\renewcommand{\P}{\operatorname{P}}
\renewcommand{\H}{\operatorname{H}}
\DeclareMathOperator{\Ap}{Ap}

\usepackage{thmtools}
\usepackage{thm-restate}

\usepackage{hyperref}
\hypersetup{
    colorlinks   = true,                    
    linkcolor    = blue!70!black,
    anchorcolor  = gray,
    citecolor    = blue!70!black,
    filecolor    = red,
    menucolor    = green,
    runcolor     = red,
    urlcolor     = [rgb]{0,0.2,0.5}
}
\usepackage{cleveref}

\usepackage[utf8]{inputenc}    

\usepackage{enumerate} 

\usepackage{comment}
\usepackage{float}
\usepackage{url}
\usepackage{xcolor}

\usepackage{pgf, tikz}
\usetikzlibrary{arrows, automata}

\usepackage{mathtools}

\usepackage{cancel}

\newcommand\pieter[1]{{\textcolor{blue}{#1}}}
\newcommand\andres[1]{{\textcolor{magenta}{#1}}}
\newcommand\alex[1]{{\textcolor{red}{#1}}}

\makeatletter
\let\@@pmod\pmod
\DeclareRobustCommand{\pmod}{\@ifstar\@pmods\@@pmod}
\def\@pmods#1{\mkern4mu({\operator@font mod}\mkern 6mu#1)}
\makeatother

\title{Cyclotomic exponent sequences of numerical semigroups}
\author[Ciolan]{Alexandru Ciolan}
\address{Max-Planck-Institut f\"ur Mathematik, Vivatsgasse 7, D-53111 Bonn, Germany}
\email{ciolan@mpim-bonn.mpg.de}

\author[Garc\'ia-S\'anchez]{Pedro A. Garc\'ia-S\'anchez}
\address{Departamento de \'Algebra, Universidad de Granada, E-18071 Granada, Espa\~na}
\email{pedro@ugr.es}
\thanks{The second author is supported by the project MTM2017--84890--P, which is funded by Ministerio de Economía y Competitividad and Fondo Europeo de Desarrollo Regional FEDER, and by the Junta de Andalucía Grant Number FQM--343.}

\author[Herrera-Poyatos]{Andr\'es Herrera-Poyatos}
\address{\parbox{\linewidth}{Department of Computer Science, University of Oxford, Wolfson Building, Parks Road, Oxford,\\OX1~3QD, UK.}}
\email{andres.herrerapoyatos@cs.ox.ac.uk}
\thanks{The third author was supported by an Initiation to Research Fellowship from the University of Granada in the academic year 2017-2018. He is currently supported by an Oxford-DeepMind Graduate Scholarship and an EPSRC Doctoral Training Partnership.}

\author[Moree]{Pieter Moree}
\address{Max-Planck-Institut f\"ur Mathematik, Vivatsgasse 7, D-53111 Bonn, Germany}
\email{moree@mpim-bonn.mpg.de}

\begin{document}
\date{}

\begin{abstract}
We  study the \textit{cyclotomic exponent sequence} of a numerical semigroup $S,$ and we
compute its values
at the gaps of $S,$  the elements of $S$ with unique representations in terms of minimal generators, and the Betti elements $b\in S$ for which the set $\{a \in \operatorname{Betti}(S) : a \le_{S}b\}$ is totally ordered with respect to $\le_S$ (we write $a \le_S b$ whenever $a - b \in S,$ with $a,b\in S$).
This allows us to characterize certain semigroup families, such as Betti-sorted or Betti-divisible numerical semigroups, as well as numerical semigroups with a unique Betti element, in terms of their cyclotomic exponent sequences. 
Our results also apply to \textit{cyclotomic numerical semigroups}, which are numerical semigroups with a finitely supported cyclotomic exponent sequence. We show that cyclotomic numerical semigroups with certain cyclotomic exponent sequences are complete intersections, thereby making progress towards proving the conjecture of Ciolan, Garc\'ia-S\'anchez and 
Moree (2016) stating that $S$ is cyclotomic if and only if it is a complete intersection.
\end{abstract}

\maketitle

\section{Introduction}

A \emph{numerical semigroup} $S$ is a submonoid of $\mathbb N$ (the set of non-negative integers) under addition, with finite complement in $\mathbb N$. The non-negative integers that are not in $S$ are its \emph{gaps}, and the set of gaps is denoted by $\operatorname{G}(S)$. The largest gap is the \emph{Frobenius number} of $S$, denoted by $\operatorname{F}(S)$. The number of gaps of $S$, also known as the \emph{genus} of $S$, is denoted by $\operatorname{g(S)}$.  A numerical semigroup admits a unique minimal generating system; its elements are called \emph{minimal generators}, and its cardinality the \emph{embedding dimension}, denoted by $\mathrm e(S)$. The smallest positive integer in $S$ is the \emph{multiplicity} of $S$ and is denoted by $\mathrm m(S)$. For an introduction to the theory of numerical semigroups the reader is referred, e.g., to \cite{ns}.  

To a numerical semigroup $S$ 
we can associate its \emph{Hilbert series},  defined as the formal power series $\operatorname{H}_S(x)=\sum_{s\in S}x^s\in \mathbb Z[\![x]\!],$ and its \emph{semigroup polynomial}, given by  $\operatorname{P}_S(x)=(1-x)\sum_{s\in S}x^s.$ (Indeed, since all elements larger than $\operatorname{F}(S)$ are in $S$ and $\operatorname{F}(S)$ is not, $\operatorname{P}_S(x)$ is a monic polynomial of degree $\operatorname{F}(S)+1.$)
In the sequel we say that a formal identity of the form $A(x)=B(x)$ is true if it holds in $\mathbb Z[\![x]\!]$. For notational convenience we will often denote the infinite sum $1+x^d+x^{2d}+\cdots$ by $(1-x^d)^{-1},$ where $d\in\mathbb N.$

 It is not difficult to conclude that the coefficients of $\operatorname{P}_S$ are in $\{-1,0,1\}$ and that consecutive non-zero coefficients alternate in sign. On noting the formal identity $\operatorname{H}_S(x)=(1-x)^{-1}-\sum_{s\in \operatorname{G}(S)}x^s$, we have
\begin{equation}
\label{eq:psx}
\operatorname{P}_S(x)=1+(x-1)\sum_{s\in \operatorname{G}(S)}x^s,
\end{equation}
and so $\operatorname{P}_S(1) = 1$. In addition, $\operatorname{P}_S(0) = 1$ and, by \cite[Lemma 11]{cyclotomic}  (see  also Lemma~\ref{lem:Wittexpansionproof}), there exist \emph{unique} integers $e_j$ such that the formal identity
\begin{equation} \label{eq:cyc-exp}
  \operatorname{P}_S(x)= \prod_{j = 1}^\infty (1 - x^j)^{e_j}
\end{equation}
holds. We call the sequence $\mathbf{e}=\{e_j\}_{j \ge1}$ the \emph{cyclotomic exponent sequence} of $S$ and we will use this notation throughout the paper.

Cyclotomic exponent sequences were introduced in \cite{cyclotomic} as a tool for studying cyclotomic numerical semigroups; we will come back to this later in this section. The purpose of this paper is to initiate the study of \emph{any} numerical semigroup by means of its cyclotomic exponent sequence. Our ultimate goal is to characterize special families of numerical semigroups in terms of properties of their cyclotomic exponent sequences.

Our first main result determines the exponent sequence of $S$ at  gaps and minimal generators.

\begin{theorem} \label{thm:ces:generators}
  If $S\ne\mathbb N$ is a numerical semigroup and $\mathbf{e}$ is its cyclotomic exponent sequence, then
  \begin{enumerate}[{\rm a)}]
      \item $e_1 = 1;$
      \item $e_j = 0$ for every $j \ge 2$ not in $S;$
      \item $e_j = -1$ for every minimal generator $j$ of $S;$
      \item $e_j = 0$ for every $j\in S$ that has only one factorization and is not a minimal generator.
  \end{enumerate}
\end{theorem}

Our second main result determines $\mathbf{e}$ at certain Betti elements (see Section~\ref{sec:pre:presentations} for a definition of the latter and the related notion of $R$-classes). In order to introduce our findings, we need the following definitions. \par Let $(X, \le)$ be a partially ordered set. We define the set $\operatorname{U}(X)$ as 
\[
\operatorname{U}(X)=\{ x\in X :\ \downarrow\! x \text{ is totally ordered}\},
\]
where $\downarrow\! x=\{ y\in X : y\le x\}$. We note that 
\begin{equation} \label{eq:U-minimals}
    \operatorname{Minimals}_{\le} X = \operatorname{Minimals}_{\le} \operatorname{U}(X).
\end{equation}

We write $a\le_S b$ if $b-a\in S$.
Since $S$ is a cancellative monoid free of units, the relation $\le_S $ defines an order relation on $\mathbb{Z}.$  Moreover, for any $s\in S$, the set $\downarrow\! s$ (considered in $(S,\le_S)$)  is finite.
\par If $S$ is a numerical semigroup, we define the set \[
\mathcal{E}(S) = \{d \in \mathbb{N} : d \ge 2,~ e_d \ne 0, ~d \text{ is not a minimal generator}\}, 
\] notation which we will use throughout.
\par Our next result relates the partially ordered sets $(\operatorname{Betti}(S), \le_S)$ and $(\mathcal{E}(S),\le_S)$.
\begin{theorem} \label{thm:ces:betti}
 Let $S$ be a numerical semigroup with cyclotomic exponent sequence $\mathbf{e}$. Then $\operatorname{U}(\operatorname{Betti}(S)) = \operatorname{U}(\mathcal{E}(S))$. Moreover, for every $b \in \operatorname{U}(\operatorname{Betti}(S))$, the exponent $e_b$ is equal to the number of $R$-classes of $b$ minus $1$.
\end{theorem}
A direct consequence of \eqref{eq:U-minimals} and Theorem~\ref{thm:ces:betti} is that $\operatorname{Minimals}_{\le_S} \operatorname{Betti}(S) = \operatorname{Minimals}_{\le_S} \mathcal{E}(S)$. In order to prove Theorem~\ref{thm:ces:betti} we need to understand the graph of factorizations $\nabla_b$ of the elements $b$ in $\operatorname{U}(\operatorname{Betti}(S))$ (see Section~\ref{sec:pre:presentations} for a definition of $\nabla_b$). Our main technical result on this matter is Theorem~\ref{thm:brf:sorted}, which shows that when $b \in \operatorname{U}(\operatorname{Betti}(S)) \setminus \operatorname{Minimals}_{\le_S} \operatorname{Betti}(S)$, the graph $\nabla_b$ has exactly one connected component that is not a singleton. 

As a consequence of Theorem~\ref{thm:ces:betti} we are able to characterize some families of numerical semigroups solely in terms of their cyclotomic exponent sequences. Before stating our next result, let us define these families. In what follows, $S$ is a numerical semigroup. We say that $S$ is \emph{Betti-sorted} if $\operatorname{Betti}(S)$ is totally ordered with respect to $\le_S,$ and that $S$ is \emph{Betti-divisible} if $\operatorname{Betti}(S)$ is totally ordered with respect to the divisibility order in $\mathbb{N}$. These two families of numerical semigroups were introduced in \cite{isolated}, where the authors showed that they are complete intersections (see Section~\ref{sec:pre:ci} for a definition). The third family  we consider is that of numerical semigroups with a unique Betti element, which is obviously a subset of each of the two previous families. This family was studied in \cite{isolated} and \cite{single-betti}.

\begin{theorem} \label{thm:ces:char}
  For a numerical semgiroup $S$ the following assertions hold:
  \begin{enumerate}[{\rm a)}]
  \item The semigroup $S$ is Betti-sorted if and only if  $\mathcal{E}(S)$ is totally ordered by $\le_S$.
  \item The semigroup $S$ is Betti-divisible if and only if $\mathcal{E}(S)$ is totally ordered by the divisibility order.
  \item The semigroup $S$ has a unique Betti element if and only if $\mathcal{E}(S)$ is a singleton.
  \end{enumerate}
\end{theorem}
Our work also has some consequences for cyclotomic numerical semigroups. 
\begin{definition}
\label{def:CNS}
A cyclotomic numerical semigroup is a numerical semigroup whose  cyclotomic exponent sequence $\mathbf{e}$ has finite support; that is, there exists $N\in\mathbb N$ such that $e_j=0$ for every $j\ge N$. 
\end{definition}
These semigroups  were introduced and studied in \cite{cyclotomic} using a different, but equivalent, definition (see  Section~\ref{sec:cyclotomicbabbling}). It turns out that every \textit{complete intersection numerical semigroup} is cyclotomic, the former being a numerical semigroup such that the cardinality of its  minimal presentation equals its embedding dimension minus one (see Section~\ref{sec:pre:presentations} for a brief recap on the concept of minimal presentations). In \cite{cyclotomic} the authors made the following conjecture,  which they checked to be true for numerical semigroups with Frobenius number not exceeding $70,$  using the \texttt{GAP} package \texttt{numericalsgps} \cite{numericalsgps,gap}. 
\begin{conjecture}[{\cite[Conjecture 1]{cyclotomic}}]
  \label{con:cyclo-ci}
  A numerical semigroup is a complete intersection if and only if it is cyclotomic.
\end{conjecture}
A version of this conjecture has been established for certain graded algebras, see \cite{cyclo-algebras}, but the numerical semigroup version remains open. Conjecture \ref{con:cyclo-ci} is equivalent with saying that \emph{a numerical semigroup $S$ is a complete intersection if and only if its cyclotomic exponent sequence $\bf{e}$ has finite support}. This establishes an equivalence between an algebraic property of a numerical semigroup and one that only involves its cyclotomic exponent sequence. Note that $\bf{e}$ has finite support if and only if $\mathcal{E}(S)$ is finite. Recall that Theorem~\ref{thm:ces:char} deals with the case where $\mathcal{E}(S)$ is a singleton.

As a consequence of our results, we make further progress towards proving Conjecture~\ref{con:cyclo-ci} by showing that all members of a certain family of cyclotomic numerical semigroups are complete intersections. More precisely, if the Hasse diagrams of $\operatorname{Betti}(S)$ and $\mathcal{E}(S)$ with respect to $\le_S$ are forests, that is, $\operatorname{U}(\operatorname{Betti}(S)) = \operatorname{Betti}(S)$ and $\operatorname{U}(\mathcal{E}(S)) = \mathcal{E}(S)$, then we are able to deduce that $S$ is a complete intersection (Corollary~\ref{cor:cyclo:forest:3}).  Computations suggest that such forests arise very frequently; for instance, there are 197 complete intersection numerical semigroups with Frobenius number 101 (equivalently, with genus equal to $52$), and for 170 of them the Hasse diagram of their set of Betti elements with respect to $\le_S$ is a forest. 
Here we should mention that, for any complete intersection numerical semigroup $S$, we have $\operatorname{Betti}(S) = \mathcal{E}(S)$, as explained in Section~\ref{sec:pre:ci}.


The paper is organized as follows. In Section~\ref{sec:pre} we gather some preliminary material used in the rest of the paper. In Section~\ref{sec:ces} we introduce cyclotomic exponent sequences and establish some elementary properties. In Section~\ref{sec:thm1} we prove Theorem~\ref{thm:ces:generators}. In Section~\ref{sec:thm2} we give the proof of Theorem~\ref{thm:ces:betti}, which comes in two parts, and we discuss a few tools needed for this purpose, such as minimal Betti elements and restricted factorizations (as this section is the longest, we kindly ask in advance for the reader's patience).  In Section~\ref{sec:thm3} we give the proof of Theorem~\ref{thm:ces:char}, while Section~\ref{sec:cns} is dedicated to applications to cyclotomic numerical semigroups, open questions, and concluding remarks.  

\section{Preliminaries} \label{sec:pre}
Here we recall a few properties and notions that are needed throughout the paper. References in the subsection headers
give suggestions for further reading.

Section~\ref{sec:cyclotomicbabbling} is exceptional in 
that it is not
needed for the rest of the paper.
Its purpose is to show that the original definition of
a cyclotomic numerical semigroup $S$, 
given in
 \cite{cyclotomic} through saying that $\P_S$ admits a factorization
into cyclotomic polynomials as in~\eqref{produkt},   
is equivalent
with the definition used here.

\subsection{Cyclotomic numerical semigroups and cyclotomic polynomials \cite{cyclotomic, cyclo:intro}}
\label{sec:cyclotomicbabbling}
The semigroup polynomial and the Frobenius number of a numerical semigroup
of embedding dimension two can be
easily determined (see, for instance, \cite{ns:cyclo-bernoulli}).
\begin{lemma}[{\cite[Theorem 1]{ns:cyclo-bernoulli}}]
\label{embedis2}
If $2\le a<b$ are coprime integers, then
\[ 
\operatorname{P}_{\langle a,b\rangle}(x)=\frac{(1-x)(1-x^{ab})}{
(1-x^a)(1-x^b)}. 
\]
\end{lemma}
\begin{corollary}[Sylvester, 1884]
\label{sylvester}
If $2\le a<b$ are coprime integers, then $\operatorname{F}(\langle a,b\rangle)=ab-a-b$.
\end{corollary}
Lemma~\ref{embedis2} shows that $S = \langle a,b\rangle$ is a 
cyclotomic numerical semigroup, since its exponent sequence has finite
support.
The factorization of $\operatorname{P}_{\langle a,b\rangle}$ 
into irreducibles is easily found 
by using the well-known factorization
\begin{equation}
  \label{eq:cyclo-rec}
  x^n - 1 = \prod_{d|n} \Phi_d(x)
\end{equation}
of $x^n-1$ into cyclotomic polynomials 
(all of them irreducible
over $\mathbb Q$).
In the special case where $a$ and $b$ are prime numbers,
we find that $\operatorname{P}_{\langle a,b\rangle}(x)=\Phi_{ab}(x)$, which
then gives a very natural proof of the classical fact that
the coefficients of $\Phi_{ab}(x)$ are all in $\{-1,0,1\}$
and that consecutive non-zero coefficients alternate in sign.
\par The following two results describe some basic properties of the cyclotomic exponent sequence attached to a cyclotomic numerical semigroup.
\begin{proposition} \label{prop:sum-e}
Let $S$ be a numerical semigroup and let $\mathbf{e}$ be its cyclotomic exponent sequence. If $S$ is cyclotomic, then 
$\sum_{j \ge 1} e_j  = 0$.
\end{proposition}
\begin{proof}
Let $N$ be the largest index $j$ such that $e_j\ne 0.$ Then we have
\[ \operatorname{P}_S(x)= (1-x)^{\sum_{j\le N}e_j}G_S(x),\] 
for some rational function $G_S(x)$ satisfying $G_S(1)\not \in \{ 0, \infty \}$
(in fact $G_S(1)=\prod_{j\le N}j^{e_j}$). Since $\operatorname{P}_S(1)=1$, it follows that $\sum_{j\ge 1}e_j=0$. 
\end{proof}

\begin{proposition}\label{prop:cyclotomic}
Let $S$ be a numerical semigroup. Then $S$ is cyclotomic if and only if $\operatorname{P}_S(x)$ factorizes in the form 
\begin{equation}
\label{produkt}
\operatorname{P}_S(x)=\prod_{d\in \mathcal{D}}\Phi_d^{h_d},
\end{equation} 
where
$\mathcal{D}$ is a finite set and $h_d$ are positive integers.
\end{proposition}
\begin{proof} 
 Let $\mathbf{e}$ be the exponent sequence of $S$ and let $N$ be the largest index $j$ such that $e_j\ne 0.$ By Proposition~\ref{prop:sum-e} we have 
 $\sum_{1\le j\le N}e_j=0$ and so
\[\operatorname{P}_S(x)=\prod_{j=1}^N (1-x^j)^{e_j}=\prod_{j=1}^N 
(x^j-1)^{e_j}.\]
By~\eqref{eq:cyclo-rec} it then follows that 
$\operatorname{P}_S(x)$ can be written as in~\eqref{produkt}, where 
a priori some of the integers $h_d$ may be negative. As the
complex zeros of the cyclotomic polynomials are all different,
this would lead to $\operatorname{P}_S$ having a pole, contradicting the fact that $\operatorname{P}_S$ is 
a polynomial.

For the other direction, we 
use the M\"obius function $\mu(n)$, which is equal to zero for non-square free integers $n$ and to $(-1)^r$ otherwise, where $r$ is the number of prime factors in the prime decomposition of $n$. By applying M\"obius inversion to~\eqref{eq:cyclo-rec} one obtains
\begin{equation*}
  \Phi_n(x)=\prod_{d \mid n}(x^d-1)^{\mu(n/d)}.
\end{equation*}
Using the fact that $\sum_{d|n}\mu(d)=0$ for $n\ge 2$, this can 
be rewritten for $n\ge 2$ as 
\begin{equation*}
\Phi_n(x)=\prod_{d \mid n}(1-x^d)^{\mu(n/d)}.
\end{equation*}
Since $\operatorname{P}_S(1)\ne 0$, $\Phi_1(1)=0$,  and $\Phi_d(1)\ne 0$ 
for $d\ge 2$, we have $1\not\in \mathcal{D}$, and
so, using the latter identity,
it follows that there are integers $e_1, e_2, \ldots$ such that
\begin{equation*}
 \operatorname{P}_S(x) = \prod_{d = 1}^{\infty} (1 - x^d)^{e_d}, 
 \end{equation*}
where $e_d = 0$ for $d > \max \mathcal{D}$, which means that 
$\mathbf{e}$ has finite support.
\end{proof}

Recall that a polynomial $f(x)$ of degree $d$ is \emph{self-reciprocal} if $f(x)=x^df(1/x)$. The cyclotomic polynomial $\Phi_n$ is self-reciprocal for $n \ge 2$. As a consequence of this fact and Proposition~\ref{prop:cyclotomic}, it follows that if $S$ is cyclotomic, then $\operatorname{P}_S(x)$ is self-reciprocal. It is not difficult to show that a numerical semigroup is \emph{symmetric}  (that is, for every $n\in\mathbb Z$, either $n$ or $\operatorname{F}(S)-n$ is in $S$) if and only if $\operatorname{P}_S$ is self-reciprocal \cite{ns:cyclo-bernoulli}. Therefore, every cyclotomic numerical semigroup is symmetric.  The converse is generally not true; for instance, it can be shown that for every positive integer $e\ge 4$, there exists a numerical semigroup of embedding dimension $e$ that is symmetric but not cyclotomic \cite{cyclo:log-deriv,SaSt}.

Proposition~\ref{prop:cyclotomic} raises the
question whether one can classify cyclotomic numerical semigroups 
for which $\operatorname{P}_S$ decomposes into a small number of irreducible factors.
This and similar questions are addressed in \cite{cyclo-length}, where the authors 
show, for example, that $\operatorname{P}_S = \Phi_n$ if and only if $n = pq$ and $S = \langle p,q \rangle$ for 
distinct prime numbers $p$ and $q$.

\subsection{Ap\'ery sets \cite{ns}} 

Let $S$ be a numerical semigroup and $m\in \Z$. The set
\[
\Ap(S;m)=\{ s\in S : s-m\not\in S\}
\]
is called the \emph{Ap\'ery set} of $m$ in $S$.  
Given any arithmetic progression modulo $m$, the numbers in it that are large enough will be
in $S$, whereas the numbers that are small enough will not be in $S.$ Therefore, among them we will find
at least one element from $\Ap(S;m)$, and so 
$|\Ap(S;m)|\ge m$. 

In the remainder of this subsection we assume that $m\in S$, in which case
$S=\Ap(S;m)+m\N$ and $|\Ap(S;m)|=m$. It then follows that every integer $z$ can be 
uniquely written as $z=km+w$ with $k\in \mathbb{Z}$ and $w\in \Ap(S;m)$, 
and that $z\in S$ if and only if $k\ge 0$. We will use this fact several times.

From $S=\Ap(S;m)+m\N$ we infer that
$\operatorname{H}_S(x)= \sum_{w \in \Ap(S; m)} x^w \sum_{k=0}^{\infty}x^{km}$,
hence 
\begin{equation} \label{eq:pol:apery}
(1 - x^m) \operatorname{H}_S(x) = \sum_{w \in \Ap(S; m)} x^w,
\end{equation}
with the right-hand side being the \emph{Ap\'ery polynomial} of $m$ in $S$, see \cite{ns:apery-hilbert}. 

\subsection{Minimal presentations and Betti elements \cite{ns-app,ns}} \label{sec:pre:presentations}

Let $S$ be a numerical semigroup minimally generated by $\{n_1,\dots, n_e\}$. There is a natural epimorphism $\varphi \colon \mathbb{N}^e\to S$, defined as $\varphi(a_1,\ldots,a_e)=\sum_{i=1}^e a_in_i$. The set $\ker\varphi=\{(a,b)\in \mathbb{N}^e\times \mathbb{N}^e \colon \varphi(a)=\varphi(b)\}$ is a congruence, that is, an equivalence relation compatible with addition; hence, $S$ is isomorphic, as a monoid, to $\mathbb{N}^e/\ker\varphi$. A \emph{presentation} for $S$ is a system of generators of $\ker\varphi$ as a congruence. A presentation is \textit{minimal} if none of its proper subsets generates $\ker\varphi$. It can be shown that all minimal presentations of a numerical semigroup have the same (finite) cardinality (see, for instance, \cite[Chapter 7]{ns}). 

 Given $\rho\subseteq \mathbb{N}^e\times \mathbb{N}^e$, denote by $\operatorname{cong}(\rho)$ the congruence generated by $\rho$, that is, the intersection of all congruences containing $\rho$. Define $\rho^0 = \rho \cup \{(y,x) : (x,y) \in \rho\}$ and $\rho^1 = \{(x+u, y+u) : (x,y) \in \rho^0, u \in \mathbb{N}^e\}$. It turns out that $\operatorname{cong}(\rho)$ is the transitive closure of $\rho^1$.

A minimal presentation of $S$ can be constructed as follows. For $s\in S$, let $\operatorname{Z}(s)$ be the \emph{set of factorizations} of $s$ in $S$, that is, the fiber $\varphi^{-1}(s)$ (we use $\ZZ(S)$ to denote the set of all factorizations of elements in $S$, which equals $\mathbb{N}^{\operatorname{e}(S)}$). Define  $\nabla_s$ to be the graph with vertices $\operatorname{Z}(s)$ and with edges $xy$ so that $x\cdot y\neq 0$ (dot product; that is, edges join factorizations having minimal generators in common). The connected components of $\nabla_s$ are called the \emph{R-classes} of $s$. The element $s\in S$ is a \emph{Betti element} if $\nabla_s$ is not connected. We denote by $\operatorname{Betti}(S)$ the set of Betti elements of $S$, and by $\operatorname{nc}(\nabla_s)$ the number of connected components of $\nabla_s$.

Assume that $s\in\operatorname{Betti}(S)$ and let $C_1,\ldots, C_r$ be the connected components of $\nabla_s$ (thus $r=\operatorname{nc}(\nabla_s)$). Pick $x_i\in C_i$ for all $i\in\{1,\ldots,r\}$, and set $\rho^{(s)}=\{(x_1,x_2),(x_2,x_3),\ldots, (x_{r-1},x_r)\}$. Then $\rho=\bigcup_{s\in\operatorname{Betti}(S)} \rho^{(s)}$ is a minimal presentation of $S$. 
All minimal presentations can be constructed by using the following idea. Think of $(x_i,x_j)$ as a link connecting $C_i$ and $C_j$. Then you need all connected components to be connected with these links. The minimal possible choice is to have a spanning tree connecting them all, once the $x_i$ have been chosen. Different choices of $x_i$ in $C_i$ and different spanning trees will yield different minimal presentations, but they all have the same cardinality (see, for instance, \cite[Chapter 4]{ns-app}).
As a consequence, all minimal presentations have cardinality equal to $\sum_{s\in\operatorname{Betti}(S)}(\operatorname{nc}(\nabla_s)-1)$.

\subsection{Complete intersection numerical semigroups \cite{ns}} \label{sec:pre:ci}

Let $S$ be a numerical semigroup with embedding dimension $e$. It can be shown that the cardinality of any minimal presentation of $S$ 
has $e-1$ as a lower bound  (see, for instance, \cite[Chapter 8]{ns}), and numerical semigroups attaining this bound are called \emph{complete intersections}. 

Let $S_1$ and $S_2$ be two numerical semigroups, and $a_1$, $a_2$ be two coprime integers such that $a_1\in S_2$, $a_2\in S_1$ and neither $a_1$, nor $a_2$ is a minimal generator. The set $a_1S_1+a_2S_2$ is a numerical semigroup known as the \emph{gluing} of $S_1$ and $S_2$. We will write $S=a_1S_1+_{a_1a_2}a_2S_2$. 

A complete intersection numerical semigroup $S$ is either $\mathbb{N}$ or a gluing $a_1S_1+a_2S_2$ with both $S_1$ and $S_2$ complete intersection numerical semigroups (see \cite[Chapter 8]{ns}). It turns out that $\operatorname{Betti}(S)=\{a_1a_2\}\cup\{ a_1b_1\colon b_1\in \operatorname{Betti}(S_1)\}\cup \{a_2b_2\colon b_2\in \operatorname{Betti}(S_2)\}$, see \cite{affine:hilbert}.

It is well-known (see \cite{affine:hilbert}) that 
\begin{equation} \label{eq:gluing:hilbert}
  \operatorname{H}_{a_1 S_1 +_{a_1 a_2} a_2 S_2}(x) = (1 - x^{a_1 a_2}) \H_{S_1}(x^{a_1}) \H_{S_2}(x^{a_2}),
\end{equation}
which, in terms of semigroup polynomials, can be written as
\begin{equation} \label{eq:gluing:pol}
\operatorname{P}_{a_1 S_1 +_{a_1 a_2} a_2 S_2}(x) = \frac{(1-x)(1 - x^{a_1 a_2})}{(1-x^{a_1})(1-x^{a_2})} \operatorname{P}_{S_1}(x^{a_1}) \operatorname{P}_{S_2}(x^{a_2}).
\end{equation}
Consequently, a formula for $\operatorname{P}_S$ in terms of the minimal generators and Betti elements of $S$ can be given. If $S=n_1\mathbb{N}+_{b_1}n_2\mathbb{N}+\cdots+_{b_{e-1}}n_e\mathbb{N}$ (with $\{n_1,\ldots,n_e\}$ the minimal generating system of $S$ and with $b_i$ not necessarily distinct integers), then \cite[Theorem 4.8]{affine:hilbert} states that
\begin{equation}\label{hilbert-gluing}
\operatorname{H}_S(x)=\frac{\prod_{i=1}^{e-1}(1-x^{b_i})}{\prod_{i=1}^e(1-x^{n_i})}.
\end{equation}
If $S=a_1S_1+_{a_1a_2} a_2S_2$ is a gluing of $S_1$ and $S_2$, then every minimal presentation of $S$ comes from the union of a minimal presentation of $S_1$, a minimal presentation of $S_2$ and a pair of factorizations of $a_1a_2$, one in $a_1S_1$ and the other in $a_2S_2$; see, for instance, \cite[Chapter 8]{ns}. Thus, by~\eqref{eq:gluing:hilbert} and the fact that every minimal presentation of $S$ has cardinality \(\sum_{b\in\operatorname{Betti}(S)}(\operatorname{nc}(\nabla_b)-1)\), the multiplicity of $b_i$ in the numerator of~\eqref{hilbert-gluing} is precisely $\operatorname{nc}(\nabla_{b_i})-1$ and the above formula can be rewritten as 
\begin{equation} \label{eq:ci:hilbert}
\operatorname{H}_S(x)=\frac{\prod_{b\in\operatorname{Betti}(S)}(1-x^b)^{\operatorname{nc}(\nabla_b)-1}}
{\prod_{i=1}^e(1-x^{n_i})}.
\end{equation}
Indeed, this identity characterizes complete intersection numerical semigroups.
\begin{proposition} \label{prop:ci-char}
Let $S$ be a numerical semigroup. Then $S$ is a complete intersection numerical semigroup if and only if $\operatorname{H}_S$ satisfies 
\eqref{eq:ci:hilbert}.
\end{proposition}
\begin{proof}
We prove that if $S$ verifies~\eqref{eq:ci:hilbert}, then $S$ is a complete intersection numerical semigroup; the other implication also holds, as we have 
just seen. Recall that $\operatorname{P}_S(x) = (1-x) \operatorname{H}_S(x)$ is a polynomial and that by~\eqref{eq:psx} we have $\operatorname{P}_S(1) = 1$. Thus the factors $1-x$ of the
numerator and denominator of $\mathrm P_S(x)$ must cancel each other out and we 
find $\sum_{b \in \operatorname{Betti}(S)}(\operatorname{nc}(\nabla_{b}) -1) = \operatorname{e}(S)-1$. Consequently, any minimal presentation of $S$ has cardinality $\operatorname{e}(S)-1$, which means that $S$ is a complete intersection.
\end{proof}

One of our aims is to prove that 
the Hilbert series of a cyclotomic numerical semigroup
always satisfies~\eqref{eq:ci:hilbert}. In this paper we 
do so for some particular 
classes of cyclotomic numerical semigroups.

\subsection{Other series and polynomials associated to numerical semigroups \cite{szekely}} \label{sec:pre:polynomials}
This subsection is dedicated to introducing a few other objects that arise naturally in connection to numerical semigroups. However, the only reults needed in the sequel are the upcoming definitions and the accompanying identity \eqref{eq:dseries}. The reader may therefore choose to omit the discussion on the polynomial $\mathcal{K}_S$, which we make here for sake of completeness, and directly skip 
to Section~\ref{sec:pre:isolated}.
\par Let $S$ be a numerical semigroup minimally generated by a set $A$. The \emph{denumerant} of $s \in S$, denoted by $\mathfrak{d}(s)$, is the cardinality of $\operatorname{Z}(s)$, the set of factorizations of $s$ in $S$. We can consider the \emph{denumerant series} $\sum_{s \in S} \mathfrak{d}(s) x^s$, which verifies the equality  
\begin{equation} \label{eq:dseries}
\sum_{s \in S} \mathfrak{d}(s) x^s = \prod_{n \in A} \sum_{j = 0}^\infty x^{jn}  = \prod_{n \in A}\frac{1}{1-x^n}.
\end{equation}
This equality is widely used in our work and its proof is straightforward. We note that every Betti element has denumerant exceeding one.

 Let $A = \{n_1 < \cdots < n_e\}$ be the minimal system of generators of $S$. 
 Sz\'ekely and Wormald \cite{szekely} were the first to study the function
\begin{equation} \label{eq:k} 
\mathcal{K}_S(x) = (1 - x^{n_1}) \cdots (1 - x^{n_e}) \H_{S}(x),
\end{equation}
which, on writing $$\mathcal{K}_S(x) = (1 + x + \dots + x^{n_1 - 1})(1-x^{n_2}) \cdots (1 - x^{n_e}) \P_{S}(x),$$ turns out to be 
a polynomial of degree $\operatorname{F}(S) + \sum_{j = 1}^{e} n_{j}$. 

Let $S$ be a complete intersection numerical semigroup. From~\eqref{eq:ci:hilbert} we derive 
\begin{equation} \label{eq:ci:k}
\mathcal{K}_S(x)=\prod_{b\in\operatorname{Betti}(S)}(1-x^b)^{\operatorname{nc}(\nabla_b)-1}.
\end{equation}

\begin{corollary}
Let $S$ be a complete intersection numerical semigroup minimally generated by $\{n_1, \dots, n_e\}$. Then
\[\operatorname{F}(S) + \sum_{j = 1}^e n_j = \sum_{b \in \operatorname{Betti}(S)} b (\operatorname{nc}(\nabla_b)-1). \]
\end{corollary}
\begin{proof}
The result follows from taking degrees in~\eqref{eq:ci:k}.
\end{proof}

The polynomial $\mathcal{K}_S$ has been explicitly computed for several families of numerical semigroups. For instance, an expression is given in \cite{free-resolutions} for numerical semigroups of embedding dimension three, and for those of embedding dimension four that are symmetric or pseudo-symmetric. In that paper, $\mathcal{K}_S$ is related to the Betti numbers of the semigroup ring associated to $S$ (see \cite{squarefree} for a different approach).

\subsection{Isolated factorizations \cite{isolated}} \label{sec:pre:isolated}

Let $S$ be a numerical semigroup and let $s\in S$. We say that a factorization $z$ of $s$ is \emph{isolated} if $z\cdot x=0$ for every factorization $x$ of $s$ different from $z$. Thus, $z$ is an isolated factorization if and only if $\{z\}$ is an $R$-class of $\nabla_s$. This means either that $s$ has a unique factorization, or that $s$ is a Betti element with one of its $R$-classes being a singleton. We denote by $\operatorname{I}(s)$ the set of isolated factorizations of $s$, and by $\operatorname{I}(\Lambda)$ the set of isolated factorizations of the elements of $\Lambda \subseteq S$. Thus
\[
\operatorname{I}(\Lambda)=\operatorname{I}_s(\Lambda)\cup \operatorname{I}_b(\Lambda),
\]
where $\operatorname{I}_s(\Lambda)$ is the set of isolated factorizations coming from elements with a unique factorization, and $\operatorname{I}_b(\Lambda)=\operatorname{I}(\Lambda)\cap\operatorname{Z}(\operatorname{Betti}(S))$ that of the isolated factorizations of the Betti elements in $\Lambda$. We also denote the cardinality of $\operatorname{I}(s)$ by $\operatorname{i}(s)$ and we define $\operatorname{IBetti}(S)$ as the set of Betti elements with an isolated factorization. Isolated factorizations can be characterized as in Lemma~\ref{lem:isolated}. First, we need some notation. Let $x, y \in \mathbb{N}^{e}$. We say that $x \le y$ if $x_j \le y_j$ for every $j$. This gives an order relation on $\mathbb{N}^{e}$, known as the \emph{cartesian product order}. Recall that $x < y$ when $x \le y$ and $x \ne y$.

\begin{lemma} \label{lem:isolated}
  Let $S$ be a numerical semigroup and $s\in S$. A factorization $z \in \operatorname{Z}(s)$ is not isolated if and only if there exists $x \in \operatorname{I}_b(S)$ such that $x < z$. In particular, 
  \[ \operatorname{I}_b(S) = \operatorname{Minimals}_{\le} \operatorname{Z}(\{s\in S : \mathfrak{d}(s) \ge 2\}). \]
  As a consequence, any factorization $z \in \mathbb{N}^{\operatorname{e(S)}}$ can be written as $z = w + x_1 + \cdots + x_l$ with $w \in \operatorname{I}_s(S)$ and $x_1, \ldots, x_l \in \operatorname{I}_b(S)$. 
\end{lemma}
\begin{proof}
  The first assertion is merely a rephrasing of \cite[Lemma 3.1]{isolated}. 
  
  Assume that $z\in \mathbb{N}^{\operatorname{e}(S)}=\operatorname{Z}(S)$. If $\varphi(z)$ has a unique factorization, then $z=w\in \operatorname{I}_s(S)$. Otherwise, there exists $x_1\in \operatorname{I}_b(S)$ such that $x_1<z$. We consider now $z-x_1$ and start anew. This process must end either with a 0 or with a factorization that is the unique factorization of an element in the semigroup. 
\end{proof}

We say that an element $s \in S$ is Betti-minimal if $s \in \operatorname{Minimals}_{\le_S} \operatorname{Betti}(S)$. As a consequence of Lemma~\ref{lem:isolated}, one can characterize Betti-minimal elements as in Proposition~\ref{prop:minimal-betti}.
\begin{proposition}[{\cite[Proposition 3.6]{isolated}}] \label{prop:minimal-betti}
  Let $S$ be a numerical semigroup and $s\in S$. The following statements are equivalent:
  \begin{enumerate}[{\rm a)}]
  \item \label{item:betti-minimal} $s$ is Betti-minimal;
  \item \label{item:ibetti-minimal} $s$ is a minimal element of $\operatorname{IBetti}(S)$ with respect to $\le_S$;
  \item \label{item:isolated} $s$ has at least two factorizations and all of them are isolated, that is, $\operatorname{nc}(\nabla_s) = \operatorname{i}(s) \ge 2$.
  \end{enumerate}
\end{proposition}

The following result is a particular case of \cite[Corollary 3.8]{isolated} and characterizes  the elements having a unique factorization in terms of Ap\'ery sets.

\begin{corollary}[{\cite[Corollary 3.8]{isolated}}] \label{cor:isolated} 
  Let $S$ be a numerical semigroup.  Then
  \[ \{m \in S \colon \mathfrak{d}(m) = 1\}\,\, =\quad \bigcap_{\mathclap{b \in \operatorname{Betti}(S)}} \operatorname{Ap}(S; b)
  \,\, =\qquad \bigcap_{\mathclap{b \in \operatorname{Minimals}_{\le_S} \operatorname{Betti}(S)}} \operatorname{Ap}(S; b). \]
\end{corollary}

The next lemma allows us to deal with sequences of Betti elements of the form $b_1 \le_S \cdots \le_S b_t$, and will be useful for the study of $\operatorname{U}(\operatorname{Betti}(S))$.

\begin{lemma}[{\cite[Lemma 3.12]{isolated}}] \label{lem:disjoint-betti}
  Let $S$ be a numerical semigroup. If $b_1$ and $b_2$ are two Betti elements of $S$ such that $b_1 <_S b_2$,
 then $x \cdot y = 0$ for every $x \in \operatorname{Z}(b_1)$ and $y \in \operatorname{I}(b_2)$. 
\end{lemma}

\section{Cyclotomic exponent sequences} \label{sec:ces}

In this section we show how to 
compute, both theoretically and practically, the cyclotomic exponent sequence of a numerical semigroup, and we give some examples.~Practically, they can be computed with the function \texttt{CyclotomicExponentSequence} or, alternatively, with the function \texttt{WittCoefficients} of the \texttt{GAP} \cite{gap} package \texttt{numericalsgps} \cite{numericalsgps}, which implements the method given in the upcoming Lemma~\ref{lem:sequence}.

Let $a(x), b(x) \in \mathbb{Z}[\![x]\!]$ 
and $p(x) \in \mathbb{Z}[x]$.
We use the notation 
$a(x) \equiv b(x) \pmod*{p(x)}$ 
to indicate that $a(x)- b(x) \in p(x) \mathbb{Z}[x]$. Note that $\equiv$ is an equivalence relation.

The next lemma, together with the
fact that $\operatorname{P}_S(x)\equiv 1\pmod*{x}$, shows that for \emph{any} numerical semigroup
there is an expansion of the form 
\eqref{eq:cyc-exp}, where
the exponents $e_j$ are uniquely determined integers.
\begin{lemma}
\label{lem:Wittexpansionproof}
Let $f(x)\in\mathbb Z[\![x]\!]$ and suppose that $f(x)\equiv 1\pmod*{x}$. Then there exist unique integers $e_1,e_2,\ldots$ such that, in $\mathbb Z[\![x]\!]$,
\begin{equation}
\label{eulerprodid}
f(x)=\prod_{k=1}^{\infty}(1-x^k)^{e_k}.
\end{equation}
\end{lemma}
\begin{proof}
We show how to successively determine the integers $e_1,e_2,\ldots,e_m$ such that 
\eqref{eulerprodid} holds modulo $x^{m+1}$. By assumption, $f(x)=1-e_1x\pmod*{x^2}$ for some $e_1\in\mathbb{Z}$. This then gives $f(x)(1-x)^{-e_1}\equiv1\pmod*{x^2}$.
Let $m\ge 2$. Suppose that we have found integers $e_1,\ldots,e_{m-1}$ such that 
\begin{equation*}
f(x)\prod_{k=1}^{m-1}(1-x^k)^{-e_k}\equiv1\pmod*{x^m}.
\end{equation*} 
As the right-hand side is of the form $1-e_mx^m\pmod*{x^{m+1}}$ for some integer $e_m$,
we infer that 
\begin{equation*}
f(x)\prod_{k=1}^{m}(1-x^k)^{-e_k}\equiv1\pmod*{x^{m+1}}.
\end{equation*}
\par We now turn our attention to the uniqueness claim. For the sake of contradiction, suppose there exists a different sequence of integers $f_n$ such that $f(x)=\prod_{k=1}^{\infty}(1-x^k)^{f_k}$. Let $m$ be the smallest integer such that $f_m\ne e_m$. Put $h(x)=\prod_{k=1}^{m-1}(1-x^k)^{e_k}$.
We then have $$f(x)\equiv h(x)\,(1-x^m)^{e_m}\pmod*{x^{m+1}}$$ on the one hand, and 
$$f(x)\equiv h(x)\,(1-x^m)^{f_m}\pmod*{x^{m+1}}$$ on the other. As the two expressions have different coefficients in front of $x^m$, we have reached a contradiction, concluding the proof.
\end{proof}
\begin{example} 
\label{cyclidentity}
Let $\alpha$ be an integer. We have
$1-\alpha x=\prod_{k=1}^{\infty}(1-x^k)^{M(\alpha,k)}$, with
$M(\alpha,k) = \frac{1}{k} \sum_{j \mid k} \mu(k/j)\alpha^j$. This
is the so-called \emph{cyclotomic identity}, see, e.g., \cite{Rota}. 
In case $p$ is a prime number, the fact that $M(\alpha,p)$ must be
an integer implies Fermat's Little Theorem stating that $\alpha^p\equiv 
\alpha \pmod*{p}$.
\end{example}
\begin{remark}
Expansions of the form  \eqref{eulerprodid} arise in quite different areas such
as automata, group, graph and Lie algebra theory; see \cite{moree:singular-series} for
some references. 
\end{remark}
\begin{remark}
\label{rem:prod}
Let $f(x)\in\mathbb Z[x]$ and suppose that $f(x)\equiv 1\pmod*{x}$.
By~\eqref{eulerprodid} we have $f(x)\equiv \prod_{k = 1}^{n} (1 - x^k)^{e_k}\pmod*{x^{n+1}}$.
This identity allows one to determine the 
first $n$ coefficients of $f$. On taking $n=\deg(f)$, we can even reconstruct $f$ completely.
\end{remark}

Although the proof of Lemma~\ref{lem:Wittexpansionproof} is constructive, it is computationally slow. If instead of a formal series we are given a polynomial, then the following 
``polynomial version" of Lemma~\ref{lem:Wittexpansionproof} provides a faster way to calculate the exponents in~\eqref{eulerprodid}. We can thus use it in the particular case
$f=\operatorname{P}_S$ in order to obtain the exponent sequence of a numerical semigroup $S$.
\begin{lemma}[{\cite[Lemma 1]{moree:singular-series}}] \label{lem:sequence}
  Let $f(x) = 1 + a_1 x + \cdots + a_d x^d \in \mathbb{Z}[x]$ be a polynomial with $a_d \ne 0$, and let $\alpha_1, \ldots, \alpha_d$ be its roots.
 Then the numbers $s_f(k) = \alpha_1^{-k} + \cdots + \alpha_d^{-k}$ 
        are integers satisfying the recursion
        \begin{equation}
            \label{eq:recursion}
            s_f(k)+a_1s_f(k-1)+\dots+a_{k-1}s_f(1)+ka_k=0,
        \end{equation}
        with $a_m=0$ for $m>d$. In particular, for $k>d$, the integer
        $s_f(k)$ is given by the linear recurrence
        \[s_f(k)=-a_1s_f(k-1)-\dots-a_{d}s_f(k-d).\]
        Over $\mathbb{Z}[\![x]\!]$ one
        has
        \[f(x)=\prod_{k = 1}^{\infty} (1 - x^k)^{e_f(k)},\]
        with 
        \begin{equation}
        \label{eq:e_fintermsofs_f}
        e_f(k) = \frac{1}{k} \sum_{j \mid k} s_f(j) \mu \left( \frac{k}{j} \right)\in \mathbb Z. 
        \end{equation}
\end{lemma}
\begin{proof}
This follows from Lemma 1 of
    \cite{moree:singular-series} on taking ${\hat F}(x)=f(x)$ and noting that
    the reciprocal $F$ of ${\hat F}$ has $\alpha_1^{-1},\ldots,\alpha_d^{-1}$ as roots. (Note that $0$ does not occur as root of $f$.)
\end{proof}
Combining Lemmas~\ref{lem:Wittexpansionproof} and~\ref{lem:sequence} with Remark~\ref{rem:prod} 
gives rise to the following.
\begin{proposition}
A numerical semigroup $S$ has 
a unique cyclotomic exponent sequence $\mathbf{e}=\{e_j\}_{j \ge1}$ 
with $e_j=e_{\operatorname{P}_S}(j)$ for every $j\ge 1$. Conversely, given 
a cyclotomic exponent sequence coming from a numerical semigroup, there is a unique
numerical semigroup corresponding to it.
\end{proposition}

Let $S$ be a complete intersection numerical semigroup. The cyclotomic exponent sequence of $S$ is provided by~\eqref{eq:ci:hilbert}. We have
\begin{enumerate}[{\rm a)}]
    \item $e_1 = 1$;
    \item $e_j = -1$ if $j$ is a minimal generator of $S$;
    \item $e_j = \operatorname{nc}(\nabla_j) -1$ if $j$ is a Betti element of $S$;
    \item $e_j = 0$ otherwise.
\end{enumerate}

\begin{example} \label{ex:ces}
As illustrated below, cyclotomic exponent sequences of cyclotomic and non-cyclotomic numerical semigroups can differ greatly in their behavior, although they may share some common properties, as observed in Theorems~\ref{thm:ces:generators} and~\ref{thm:ces:betti}.
\begin{enumerate}[{\rm a)}]
    \item The semigroup $S=\langle 4,6,9\rangle$ is a complete intersection, hence cyclotomic. Indeed, we have 
\begin{align*}
\operatorname{P}_S(x)&=x^{12}-x^{11}+x^8-x^7+x^6-x^5+x^4-x+1\\
&=(1-x)(1-x^4)^{-1}(1-x^6)^{-1}(1-x^9)^{-1}(1-x^{12})(1-x^{18}),
\end{align*}
and the cyclotomic exponent sequence of $S$ is given by \[1,0,0,-1,0,-1,0,0,-1,0,0,1,0,0,0,0,0,1, 0, \ldots.\] 

\item \label{item:ex:b} Let $S=\langle 3,5,7\rangle,$ with semigroup polynomial $\operatorname{P}_S(x)=x^5-x^4+x^3-x+1.$ 
The first 100 entries of its cyclotomic exponent sequence are  
\begin{verbatim}
 1, 0, -1, 0, -1, 0, -1, 0, 0, 1, 0, 1, 0, 1, 0, 0, -1, 0, -1, 0, 0, 1, 0, 
  1, 0, 1, -1, 0, -2, 0, -2, 1, -1, 3, 0, 3, -1, 3, -3, 1, -5, 1, -5, 3, -3, 
  7, -2, 8, -4, 7, -9, 4, -14, 6, -14, 12, -10, 22, -9, 25, -16, 23, -30, 17, 
  -42, 23, -43, 41, -36, 66, -37, 76, -60, 73, -100, 66, -133, 91, -139, 148, 
  -129, 219, -146, 252, -222, 252, -340, 255, -438, 346, -469, 524, -473, 
  731, -564, 846, -820, 887, -1183, 973, -1488, 1309, -1635, 1889, -1756, 
  2530, -2157, 2947, -3026, 3214, -4181, 3701, -5187, 4922, -5839, 6834, 
  -6563, 8905, -8200, 10467, -11195, 11807, -14992, 14052, -18463, 18510, 
  -21237, 24982, -24675, 31960, -31101, 37904, -41573, 43905, -54450, 53343, 
  -66840, 69606, -78312, 91968, -93176, 116272, -117909, 139142, -155059, 
  164573, -199918, 202659, -245305, 262345,
\end{verbatim}
which suggests that $S$ is not cyclotomic, and this is indeed the case (for otherwise, if  $S$ were cyclotomic, the roots of $\operatorname{P}_S$ would be of absolute value 1; since $\deg (\operatorname{P}_S)=5$, we would have at least one real root, which can only be $\pm1$, a contradiction).
\end{enumerate}
 \end{example}

One can rapidly check with the help of a computer that the roots of $\operatorname{P}_S(x)$ in Example~\ref{ex:ces}\ref{item:ex:b} satisfy the hypothesis of Corollary~\ref{cor:growth}, hence explaining why the cyclotomic exponents grow exponentially.

\begin{corollary}
\label{cor:growth}
Suppose that the roots $\alpha_i$ of $f$ are ordered in such a way that
their absolute value is non-decreasing and, in addition, 
$|\alpha_1|<|\alpha_2|$. Then we have
\[ \left|e_f(k)-\frac{\alpha_1^{-k}}{k}\right| \le \frac{d(k)}{k}(|\alpha_1|^{-k/2}+ {\rm deg}(f)\,|\alpha_2|^{-k}), \]
with $d(k)$ the number of divisors of
$k$. In case $f$ is linear, the term involving $\alpha_2$ can 
be omitted.
\end{corollary}
\begin{proof}\
From ~\eqref{eq:e_fintermsofs_f} and
$s_f(j) = \alpha_1^{-j} + \cdots + \alpha_{\text{deg}(f)}^{-j}$, 
we infer that
$e_f(k)=\frac{1}{k}\alpha_1^{-k}+E(k)$ with
$$k|E(k)|\le \sum_{j\mid k,\,j<k}|\alpha_1|^{-j}+\sum_{r=2}^{\text{deg}(f)}\sum_{j\mid k}|\alpha_r|^{-j}\le d(k)|\alpha_1|^{-k/2}+{\rm deg}(f)\,d(k)\,|\alpha_2|^{-k}.$$
The double sum does not arise in case $f$ is linear.
\end{proof}

\begin{example}
Let $\alpha=p$ be a prime number in 
Example~\ref{cyclidentity}. Gauss already knew that 
$M(p,k)$ is the number of irreducible monic polynomials of degree
$k$ in the ring $\mathbb F_p[x]$. By 
Corollary~\ref{cor:growth} we have
$$\frac{p^k}{k}-\frac{2}{k}p^{k/2}\le M(p,k)\le \frac{p^k}{k}+\frac{2}{k}p^{k/2}.$$
Regarded as ``building blocks," irreducible monic polynomials
are for $\mathbb F_p[x]$ what prime numbers are for $\mathbb Z.$ In this sense, the above estimate can be compared with the Prime Number
Theorem, which  states that the number of primes $p\le x$ is asymptotically equal to
$x/\log x$.
\end{example}

\section{Cyclotomic exponent sequences, gaps and minimal generators} \label{sec:thm1}
In this section we prove Theorem~\ref{thm:ces:generators} by combining the upcoming Lemmas~\ref{lem:ces:1},~\ref{lem:ces:gaps},~\ref{lem:ces:minimal-generators} and~\ref{lem:ces:one-factorization}. We start with the trivial observation that $S = \mathbb{N}$ if and only if $\operatorname{P}_S(x) = 1$. That is, $\mathbb{N}$ is the only numerical semigroup whose cyclotomic exponent sequence is constantly zero.  We deal with the case $S \ne \mathbb{N}$ for the remainder of this section.

Throughout this section we let $A$ be the minimal generating system of $S$, and $\mathbf{e}=\{e_j\}_{j\ge1}$ its cyclotomic exponent sequence. 

\begin{lemma}\label{lem:ces:1}
For every $j\in \{2,\ldots,\operatorname{m}(S)-1\}$, we have $e_j=0$. Moreover, $e_1=1$ and $e_{\operatorname{m}(S)}=-1$.
\end{lemma}
\begin{proof}
Let $m = \operatorname{m}(S)$. From the congruence
$\operatorname{H}_S(x)\equiv 1+x^m \pmod*{x^{m+1}}$ we see
that 
\begin{equation*}
  \operatorname{P}_S(x) \equiv 1-x+x^m\equiv (1-x)(1+x^m) \equiv (1-x)(1-x^m)^{-1} \pmod*{x^{m+1}}.
\end{equation*}
The latter congruence determines $e_1,\ldots,e_m$ uniquely, see 
the proof of Lemma~\ref{lem:Wittexpansionproof}.
\end{proof}

It follows that we have the formal identity
\begin{equation} \label{eq:hs}
    \operatorname{H}_S(x) = \prod_{j = \operatorname{m}(S)}^\infty (1 - x^j)^{e_j}.
\end{equation}

\begin{lemma} \label{lem:ces:gaps}
  For every gap $g$ of $S$ with $g>1$, we have $e_g=0$.
\end{lemma}
\begin{proof}
We proceed by contradiction. Let $g>1$ be the smallest gap of $S$ having a non-zero exponent. 
Considering the series $\operatorname{H}_S(x)$ modulo $x^{g+1}$, we obtain 
\begin{equation} \label{eq:hi:gap}
    \begin{aligned} 
    \operatorname{H}_S(x) & \equiv \left(1 - x^{g}\right)^{e_{g}} \prod_{j = \operatorname{m}(S)}^{g-1} (1 - x^{j})^{e_{j}} \\ 
                    & \equiv -e_g x^g + \prod_{j = \operatorname{m}(S)}^{g-1} (1 - x^{j})^{e_{j}} \pmod*{x^{g+1}}.
    \end{aligned}
\end{equation}
Since, by assumption, we have $e_j = 0$ for every gap $j$  of $S$ with $2 \le j < g$, the series expansion of $\prod_{j = \operatorname{m}(S)}^{g-1} (1 - x^{j})^{e_{j}}$ is of the form $\sum_{s \in S} a_s x^s$, and 
so in particular the coefficient of $x^g$ is zero. 
On comparing the coefficients of $x^g$ in both sides of~\eqref{eq:hi:gap}, we  obtain $0 = - e_g$, a contradiction.
\end{proof}

\begin{lemma} \label{lem:ces:minimal-generators}
For every $n\in A$ we have $e_n=-1.$
\end{lemma}
\begin{proof}
Let $n$ be a minimal generator of $S$. By considering $\operatorname{H}_S(x)$ modulo $x^{n+1}$, we have
\begin{equation} \label{eq:hi:minimal-generator}
    \begin{aligned} 
    \operatorname{H}_S(x) & \equiv (1 - x^{n})^{e_{n}} \prod_{j = \operatorname{m}(S)}^{n-1} (1 - x^{j})^{e_{j}}  \\ 
                    & \equiv -e_n x^n  + \prod_{j = \operatorname{m}(S)}^{n-1} (1 - x^{j})^{e_{j}} \pmod*{x^{n+1}}.
    \end{aligned}
\end{equation}
By Lemma~\ref{lem:ces:gaps}, the series expansion of $\prod_{j = \operatorname{m}(S)}^{n-1} (1 - x^{j})^{e_{j}}$ is of the form $\sum_{s \in S} a_s x^s$. Since $n$ cannot be written as a sum of 
two or more non-zero elements of $S$, we have $a_n = 0$.
On comparing the coefficients of  $x^n$ in both sides of~\eqref{eq:hi:gap}, we obtain $1 = - e_n$.
\end{proof}

\begin{remark}\label{rem:hs}
Let $S$ be a numerical semigroup minimally generated by $A$, and let $\mathbf{e}$ be its cyclotomic exponent sequence. 
By Lemmas~\ref{lem:ces:1},~\ref{lem:ces:gaps} and~\ref{lem:ces:minimal-generators}, we have $\{1\}\cup A \subseteq \{j \in \mathbb{N} : e_j \ne 0\}$ and $\{j \in \mathbb{N} : e_j \ne 0,~ j \ge 2\} \subseteq S$. Hence, 
\begin{equation*}
    \mathcal{E}(S) = \{j \in \mathbb{N} : e_j \ne 0, ~j \ge 2\} \setminus A
    \subseteq S.
\end{equation*}
In light of~\eqref{eq:dseries}, we obtain
\begin{equation}\label{eq:cyc-exp-H}
\operatorname{H}_S(x) = \sum_{s \in S} \mathfrak{d}(s) x^s \prod_{d\in \mathcal{E}(S)}(1-x^d)^{e_d}.
\end{equation}
\end{remark}

\begin{lemma} \label{lem:ces:one-factorization}
   For every $d \in \mathcal{E}(S)$, we have $\mathfrak{d}(d)\ge 2$. Moreover, if $\alpha \in \operatorname{Minimals}_{\le_S} \mathcal{E}(S)$, then $e_\alpha = \mathfrak{d}(\alpha) - 1$.
\end{lemma}
\begin{proof}
  Let $d \in \mathcal{E}(S)$ and $\alpha \in \operatorname{Minimals}_{\le_S} \mathcal{E}(S)$ with $\alpha \le_S d$. By comparing the coefficients of $x^\alpha$ in both sides of~\eqref{eq:cyc-exp-H} we find that $1 = \mathfrak{d}(\alpha) - e_\alpha$. Since $1 \le \mathfrak{d}(\alpha)$ and $e_\alpha \ne 0$, we obtain $e_\alpha > 0$. In particular, we find that $\mathfrak{d}(\alpha) = 1+e_\alpha \ge 2$ and, thus, $\mathfrak{d}(d) \ge 2$.
\end{proof}

\begin{proof}[Proof of Theorem~\ref{thm:ces:generators}]
The first statement in Lemma~\ref{lem:ces:one-factorization} is equivalent to the fact that $e_j=0$ for every $j\in S\setminus A$ with $\mathfrak{d}(j)=1$. Theorem~\ref{thm:ces:generators} now follows by combining Lemmas~\ref{lem:ces:1}--\ref{lem:ces:minimal-generators} and~\ref{lem:ces:one-factorization}. 
\end{proof}
We illustrate this theorem with an example.
\begin{example}
  Let $S = \langle 3, 5, 7 \rangle$ be the numerical semigroup considered in Example~\ref{ex:ces}b. We have $\operatorname{G}(S) = \{1,2,4\}$ and $\operatorname{Betti}(S) = \{10,12,14\}$. Moreover, the elements of $S$ that are not minimal generators, but yet have only one factorization, are $6,8,9,11$. Thus we have $e_j = 0$ for $j \in \{2,4,6,8,9,11\}$, which is consistent with the entries of the exponent sequence given in Example~\ref{ex:ces}b.
\end{example}

\section{Cyclotomic exponent sequences and Betti elements} \label{sec:thm2}

The goal of this section is to prove Theorem~\ref{thm:ces:betti}. This result relates the partially ordered sets $(\operatorname{Betti}(S), \le_S)$ and $(\mathcal{E}(S), \le_S)$. Let $(X, \le)$ be a partially ordered set. In the introduction we defined the set $\operatorname{U}(X)$ as 
\[
\operatorname{U}(X)=\{ x\in X :\ \downarrow\! x \text{ is totally ordered}\},
\]
where $\downarrow\! x=\{ y\in X : y\le x\}$. Note that $\operatorname{Minimals}_{\le} X$ is a subset of $\operatorname{U}(X)$. 
Recall that Theorem~\ref{thm:ces:betti} states that $\operatorname{U}(\operatorname{Betti}(S)) = \operatorname{U}(\mathcal{E}(S))$, and that it also determines the cyclotomic exponents of these Betti elements.\par 
The proof of Theorem~\ref{thm:ces:betti} is carried out by induction and is rather technical and elaborate. 
For the benefit of the reader, it is divided into several parts. We start  by showing, in Section~\ref{sec:thm2:minimal}, that $\operatorname{Minimals}_{\le_s} \operatorname{Betti}(S) = \operatorname{Minimals}_{\le_S} \mathcal{E}(S)$.   The proof of this result is simple and elegant, and serves as motivation and warm-up for the work carried out in this section. Moreover, the fact that $\operatorname{Minimals}_{\le_s} \operatorname{Betti}(S) = \operatorname{Minimals}_{\le_S} \mathcal{E}(S)$ is used in the base case of the induction of the proof of Theorem~\ref{thm:ces:betti}. In Section~\ref{sec:thm2:proof-outline} we explain the main idea behind this induction. In Section~\ref{sec:thm2:brf} we study the graph $\nabla_b$ of factorizations of the elements $b \in \operatorname{U}(\operatorname{Betti}(S))$ and develop the technical results that we need to tackle Theorem~\ref{thm:ces:betti}. Finally, in Section~\ref{sec:thm2:thm2} we use our insights on $\nabla_b$ to complete the proof of Theorem~\ref{thm:ces:betti}.

\subsection{Minimal Betti elements} \label{sec:thm2:minimal}
In this section we relate the Betti-minimal elements of a numerical semigroup to its cyclotomic exponent sequence. The following result  was established in \cite{cyclotomic}.
\begin{lemma}[{\cite[Lemma 15]{cyclotomic}}] \label{lem:betti1}
If $S$ is a cyclotomic numerical semigroup such that $\{d : e_d < 0\}$ is its minimal system of generators and $\max \{d : e_d < 0\} < \min \mathcal{E}(S),$ then $\min \mathcal{E}(S) =\min \operatorname{Betti}(S)$. 
\end{lemma}

In order to prove  \cite[Lemma 15]{cyclotomic} the authors show that $\min \mathcal{E}(S)$ is the smallest element of $S$ having at least two representations in terms of the minimal generators of $S$. Here we reach a much more general conclusion without any assumptions on $S$.
    
Let $m$ and $b$ be two positive integers. The Taylor series centered at $0$ of the complex function $(1-z^b)^{-m}$ is given by
\begin{equation} \label{eq:taylor}
  (1-z^b)^{-m} = \sum_{j = 0}^\infty \binom{m+j-1}{j} z^{jb} = \bigg(\sum_{j = 0}^\infty z^{jb}\bigg)^m
\end{equation}
and its radius of convergence is $1$. Therefore, the expression~\eqref{eq:cyc-exp-H} for $\operatorname{H}_S(x)$ can be rewritten as
\begin{equation}\label{eq:hs:b:expanded}
  \operatorname{H}_S(x) =\sum_{s \in S} \mathfrak{d}(s)x^s \prod_{\substack{d \in \mathcal{E}(S) \\ e_d < 0}} \ \sum_{j = 0}^\infty \binom{-e_d+j-1}{j} x^{j d}\prod_{\substack{d \in \mathcal{E}(S) \\ e_d > 0}} \ \sum_{j=1}^{e_{d}}\binom{e_{d}}{j} (-1)^j x^{j d}.
\end{equation}

\begin{proposition}\label{prop:den2-omega-below}
  Let $S$ be a numerical semigroup. For every $s\in S$ with $\mathfrak{d}(s) \ge 2$, there exists $d \in \mathcal{E}(S)$ such that $d \le_S s$.
\end{proposition}
\begin{proof}
  From~\eqref{eq:hs:b:expanded} we infer that if there is no $d \in \mathcal{E}(S)$ with $d \le_S s$ for some $s\in S$, then $\mathfrak{d}(s)=1$. 
\end{proof}

\begin{theorem}\label{thm:betti-minimals:omega}
  Let $S$ be a numerical semigroup. Then 
  $\operatorname{Minimals}_{\le_S} \operatorname{Betti}(S) = \operatorname{Minimals}_{\le_S} \mathcal{E}(S)$.
  Moreover, we have $e_\alpha =  \mathfrak{d}(\alpha) -1 = \operatorname{i}(\alpha) -1$ for every $\alpha\in \operatorname{Minimals}_{\le_S} \mathcal{E}(S)$.
\end{theorem}
\begin{proof}
  Let $\beta$ be a Betti element minimal with respect to $\le_S$. Then $\mathfrak{d}(\beta)\ge 2$ and, by Proposition~\ref{prop:den2-omega-below}, there exists $\alpha \in \operatorname{Minimals}_{\le_S} \mathcal{E}(S)$ such that $\alpha \le_S \beta$. Since $\mathfrak{d}(\alpha) \ge 2$ by Lemma~\ref{lem:ces:one-factorization}, Corollary~\ref{cor:isolated} states that there exists a Betti element $\beta'$ such that $\beta' \le_S \alpha$. The minimality of $\beta$ forces $\beta' = \alpha = \beta$. Hence, we have $\beta \in \operatorname{Minimals}_{\le_S} \mathcal{E}(S)$. The other inclusion is proved similarly. The value $e_\alpha$ is found by invoking Lemma~\ref{lem:ces:one-factorization} and Proposition~\ref{prop:minimal-betti}. 
\end{proof}

\subsection{Proof idea of Theorem~\ref{thm:ces:betti}} \label{sec:thm2:proof-outline}

In this section we describe our approach to proving Theorem~\ref{thm:ces:betti}. We focus on the inclusion $\operatorname{U}(\mathcal{E}(S)) \subseteq \operatorname{U}(\operatorname{Betti}(S))$.

Let $\eta \in \operatorname{U}(\mathcal{E}(S))$ and $\Lambda = \downarrow\! \eta$. We want to show that $\eta \in \operatorname{U}(\operatorname{Betti}(S))$. Since $\eta \in \operatorname{U}(\mathcal{E}(S))$, we can write $\Lambda = \{b_1 \le_S \ldots \le_S b_l\}$ with $b_l = \eta$. Note that $b_1 \in \operatorname{Minimals}_{\le_S} \mathcal{E}(S)$, so $b_1$ is Betti-minimal by Theorem~\ref{thm:betti-minimals:omega}.  Let $i$ be an integer with $2 \le i \le l$. We define
  \begin{equation}\label{eq:ri:1}
    \sum_{s \in S} r_{i-1}(s)x^s = \operatorname{H}_S(x) \prod_{j = 1}^{i-1} (1 - x^{b_j})^{-e_{b_j}}.
  \end{equation}
In light of~\eqref{eq:cyc-exp-H}, we have
  \begin{equation*} 
    \sum_{s \in S} r_{i-1}(s)x^s = \sum_{s \in S} \mathfrak{d}(s) x^s\prod_{d \in \mathcal{E}(S) \setminus \Lambda_{i-1}} (1 - x^d)^{e_d}.
  \end{equation*}
  We can rewrite this expression as 
    \begin{equation} \label{eq:ri:2}
    \sum_{s \in S} r_{i-1}(s)x^s = \sum_{s \in S} \mathfrak{d}(s) x^s \sum_{k \in \langle \mathcal{E}(S) \setminus \Lambda_{i-1} \rangle} c_k x^k,
  \end{equation}
  for certain coefficients $c_k$ with $c_0 = 1$ and $c_d = - e_d$ for any $d \in \operatorname{Minimals}_{\le_S} \left( \mathcal{E}(S) \setminus \Lambda_{i-1} \right)$. Therefore, we have
  \begin{equation} \label{eq:ri:values}
      \begin{aligned}
        r_{i-1}(s) &= \mathfrak{d}(s)  \qquad \qquad \, \text{when } \{d \in \mathcal{E}(S) : d \le_S s\} \subseteq \Lambda_{i-1}, \\
        r_{i-1}(s) &= \mathfrak{d}(s) - e_s \qquad \text{when } s \in \operatorname{Minimals}_{\le_S} \left(\mathcal{E}(S) \setminus \Lambda_{i-1}\right).
        \end{aligned}
  \end{equation}
  Most of our work is devoted to finding an algebraic interpretation of the coefficients $r_{i-1}(s)$. Assuming that $b_1, \ldots, b_{i-1} \in \operatorname{U}(\operatorname{Betti}(S))$, we want to show by induction on $i$ that, when $s \in \operatorname{Minimals}_{\le_S} \left(\mathcal{E}(S) \setminus \Lambda_{i-1} \right)$, we have $0 < r_{i-1}(s)$ and $r_{i-1}(s)$ corresponds to the number of factorizations of $s$ that are not isolated. Combining this assertion with~\eqref{eq:ri:values} for $s = b_i$, we will conclude that in this situation $e_{b_i}$ is the number of isolated factorizations of $b_i$, so $b_i$ is a Betti element (recall here that $\mathfrak{d}(b_i) \ge 2$ by Corollary~\ref{cor:isolated}). 
  In this induction, a key role is played by the observation that 
  \begin{equation*}
    \sum_{s \in S} r_{i}(s)x^s = (1 - x^{b_i})^{-e_{b_i}} \sum_{s \in S} r_{i-1}(s)x^s 
  \end{equation*}
  can be rewritten as
  \begin{equation*}
  \sum_{s \in S} r_{i}(s)x^s  = \sum_{j = 0}^\infty \binom{\operatorname{i}(b_i) + j - 1}{j} \sum_{s \in S} r_{i-1}(s)x^s,
  \end{equation*}
  where we use $e_{b_i} = \operatorname{i}(b_i)$ and~\eqref{eq:taylor}. Hence, for any $s \in S$, we obtain
  \begin{equation} \label{eq:ri:recurrence}
      r_i(s) = \sum_{j = 0}^{q_s} r_{i-1}(s - j b_i) \binom{\operatorname{i}(b_i) + j - 1}{j},
  \end{equation}
    where $q_s$ is the largest integer such that $s - q_s b_i \in S$. 
    
    In order to study the connection between the coefficient $r_{i}(s)$ and the factorizations of $s$, we introduce \emph{Betti restricted factorizations} in Section~\ref{sec:thm2:brf}, and we show that the number of Betti restricted factorizations of certain elements satisfies a similar recursion to~\eqref{eq:ri:recurrence}. In order to be 
    able to fill in the details of this proof, we will also need to understand the graph of factorizations $\nabla_b$ of any $b \in \operatorname{U}(\operatorname{Betti}(S))$.
    
\subsection{Betti restricted factorizations and $\operatorname{U}(\operatorname{Betti}(S))$} \label{sec:thm2:brf}

The main goal of this subsection is to prove Theorem~\ref{thm:brf:sorted}, which describes the graph of factorizations $\nabla_b$ of any $b \in \operatorname{U}(\operatorname{Betti}(S))$, and to develop the machinery that we need to complete the proof presented in Section~\ref{sec:thm2:proof-outline}. To this end, we introduce the concept of Betti restricted factorizations and we study their properties.

 Let $S$ be a numerical semigroup and let $\Lambda \subseteq \operatorname{IBetti}(S)$. We define the \emph{set of Betti restricted factorizations} of an element $s \in S$ with respect to $\Lambda$ by
\[ \operatorname{B}(s; \Lambda) = \left\{w + x_1 + \dots + x_l \in \operatorname{Z}(s) : w \in \operatorname{I}_s(S), ~ l \ge 0 \text{ and } x_1, \ldots, x_l \in \operatorname{I}(\Lambda)\right\}.\]
If $\Lambda = \{b\}$, then we use the notation $\operatorname{B}(s; b) = \operatorname{B}(s; \Lambda)$. Note that $\operatorname{B}(s; \emptyset) = \emptyset$ when $s$ has at least two factorizations. Furthermore, if $s$ only has one factorization, then $\operatorname{B}(s; \Lambda) = \operatorname{Z}(s)$. Another trivial observation is that if $\Lambda_1 \subseteq \Lambda_2$, then $\operatorname{B}(s; \Lambda_1) \subseteq \operatorname{B}(s; \Lambda_2)$. 

Betti restricted factorizations  allow us to obtain some information about the number of isolated factorizations of certain Betti elements, and play an essential role in the proof of Theorem~\ref{thm:ces:betti}. First we need the following result, which shows that $\operatorname{B}(s; \Lambda)$ and $\operatorname{I}_b(S)$ are closely related.

\begin{lemma} \label{lem:brf:iso}
  Let $S$ be a numerical semigroup. Let $\Lambda \subseteq \operatorname{IBetti}(S)$ and $s \in S\setminus \Lambda$. Then $\operatorname{B}(s; \Lambda) \subseteq \operatorname{Z}(s) \setminus \operatorname{I}_b(S)$. Moreover, if $\{b \in \operatorname{IBetti}(S) : b <_S s\} \subseteq \Lambda$, then $\operatorname{B}(s; \Lambda) = \operatorname{Z}(s) \setminus \operatorname{I}_b(S)$.
\end{lemma}
\begin{proof}
  If $\mathfrak{d}(s) = 1$, then $\operatorname{Z}(s)\cap \operatorname{I}_b(S) = \emptyset$ and the result follows, thus we may assume that $\mathfrak{d}(s) \ge 2$. Let $z \in \operatorname{B}(s; \Lambda)$. There is $x \in \operatorname{I}(\Lambda)$ such that $x < z$. By Lemma~\ref{lem:isolated}, $z$ is not isolated, that is, $z \in \operatorname{Z}(s) \setminus \operatorname{I}_b(S)$. Finally, if $\{b \in \operatorname{IBetti}(S) : b <_S s\} \subseteq \Lambda$, then Lemma~\ref{lem:isolated} asserts that every $z \in \operatorname{Z}(s) \setminus \operatorname{I}_b(S)$ can be expressed as an element of $\operatorname{B}(s; \Lambda)$.
\end{proof}

If $s$ has at least two factorizations and $\operatorname{B}(s; \Lambda) = \operatorname{Z}(s) \setminus \operatorname{I}_b(S)$, then we find that the number of connected components of $\nabla_{s}$ equals $\operatorname{i}(s)$ plus the number of $R$-classes of $\operatorname{B}(s; \Lambda)$, where $\operatorname{i}(s)$ is the number of isolated factorizations of $s$. Hence, in this context, $\operatorname{B}(s; \Lambda)$ is connected in $\nabla_s$ if and only if $\operatorname{i}(s) = \operatorname{nc}(\nabla_s)-1$. This observation plays an important role in the proof of Theorem~\ref{thm:brf:sorted}. First, we prove that $\operatorname{B}(s; b)$ is always connected.

\begin{lemma} \label{lem:brf:1}
Let $S$ be a numerical semigroup and let $b \in \operatorname{IBetti}(S)$. For every $s \in S$ write $s = \omega_s + q_s b$, with $\omega_s \in \Ap(S; b)$ and $q_s \in \mathbb{N}$. 
\begin{enumerate}[{\rm a)}]
\item If $\operatorname{Z}(\omega_s) = \{w\}$, then
\begin{equation*}
\operatorname{B}(s; b) = \left\{w + x_1 + \cdots + x_{q_s}  : x_1, \ldots, x_{q_s} \in \operatorname{I}(b)\right\}.
\end{equation*}
\item If $\mathfrak{d}(\omega_s) \ge 2$, then $\operatorname{B}(s; b) = \emptyset$.
\end{enumerate}
\end{lemma}
\begin{proof}
Let us assume that $\operatorname{B}(s; b) \ne \emptyset$. Let $w + x_1 + \cdots + x_l \in \operatorname{B}(s; b)$ with $w \in \operatorname{I}_s(S)$ and $x_1, \ldots, x_l \in \operatorname{I}(b)$. By Corollary~\ref{cor:isolated}, we have $\varphi(w) \in \Ap(S; b)$, and thus $s = \varphi(w) + l b$. Hence, we obtain $\varphi(w) = \omega_s$ and $q_s = l$. Since $w \in \operatorname{I}_s(S)$, we infer that $\operatorname{Z}(\omega_s) = \{w\}$ and 
\begin{equation*}
\operatorname{B}(s; b) = \left\{w + x_1 + \cdots + x_{q_s} : x_1, \ldots, x_{q_s} \in \operatorname{I}(b)\right\}.
\end{equation*}
Consequently, if $\mathfrak{d}(\omega_s) \ge 2$, then $\operatorname{B}(s; b) = \emptyset$. Finally, if $\operatorname{Z}(\omega_s) = \{w\}$, then we have $w + q_s x \in \operatorname{B}(s; b)$ for any $x \in \operatorname{I}(b)$, which implies that $\operatorname{B}(s; b)$ is not empty, and the result follows.
\end{proof}

The base case of the induction presented in Section~\ref{sec:thm2:proof-outline} requires us to prove that $r_1(s) = \lvert\operatorname{B}(s, b_1)\rvert$ for some elements $s$. 
This identity follows from the following corollary.

\begin{corollary} \label{cor:brf:1}
  Let $S$ be a numerical semigroup and let $b \in \operatorname{IBetti}(S)$. For every $s \in S$ we have
  \[
     \lvert\operatorname{B}(s; b)\rvert = \begin{cases}
                                        \binom{\operatorname{i}(b)+q_s-1}{q_s} & \text{if } \mathfrak{d}(\omega_s) = 1; \\ 
                                        0 & \text{otherwise},
                                    \end{cases}
  \]
  where $q_s$ and $\omega_s$ are as in Lemma~\ref{lem:brf:1}.
\end{corollary}
\begin{proof}
 The result is a consequence of Lemma~\ref{lem:brf:1}. Let $s \in S$. If $\mathfrak{d}(\omega_s) \ge 2$, then $\lvert\operatorname{B}(s; b)\rvert = 0$. Otherwise, since the isolated factorizations of $b$ are disjoint, we find that every element of $\operatorname{B}(s; b)$ is uniquely determined by $q_s$ isolated factorizations of $b$. The proof is completed by counting the number of combinations with repetitions of size $q_s$ from a set with $\operatorname{i}(b)$ elements.
\end{proof}

We will use the following observation several times in the proof of Lemma~\ref{lem:brf:sorted}, which shows connectivity of $\operatorname{B}(s; \Lambda)$ in $\nabla_s$ under some hypotheses.

\begin{lemma} \label{lem:brf:existence}
  Let $S$ be a numerical semigroup and let $b_1$ be a Betti-minimal element of $S$. Let $\Lambda \subseteq \operatorname{IBetti}(S)$ with $b_1\in \Lambda$. For each $s \in S$ such that $b_1$ is the only Betti-minimal element $b$ of $S$ with $b \le_S s$, and for each $x \in \operatorname{I}(\Lambda)$ such that $\varphi(x) \le_S s$, there exists $z \in \operatorname{B}(s; \Lambda)$ with $x < z$.
\end{lemma}
\begin{proof}
  Let $s \in S$ be such that $b_1$ is the only Betti-minimal element of $S$ below $s$ with respect to $\le_S$, and let $x \in \operatorname{I}(\Lambda)$ such that $\varphi(x) \le_S s$. Set $b = \varphi(x)$. Write $s-b = \omega  + qb_1$, where $\omega \in \Ap(S; b_1)$ and $q \in \mathbb{N}$. By Corollary~\ref{cor:isolated}, $\omega$ has only one factorization. We can choose $z = w+x+qy \in \operatorname{B}(s; \Lambda)$, where $\operatorname{Z}(\omega) = \{w\}$ and $y \in \operatorname{I}(b_1)$. Here we have used that Betti-minimal elements have isolated factorizations by Proposition~\ref{prop:minimal-betti}.
\end{proof}

Recall that $\operatorname{U}(\operatorname{Betti}(S))$ is the set of $b\in \operatorname{Betti}(S)$ such that $\downarrow\! b=\{ b'\in \operatorname{Betti}(S) : b'\le_S b\}$ is totally ordered. In Lemma~\ref{lem:brf:sorted} we establish connectivity of $\operatorname{B}(s; \Lambda)$ under some assumptions. This is the main ingredient of the proof of Theorem~\ref{thm:brf:sorted}.

\begin{lemma} \label{lem:brf:sorted}
  Let $S$ be a numerical semigroup. Let $u \in \operatorname{U}(\operatorname{Betti}(S))$ and 
  let $\Lambda = \downarrow\! u$. If $s \in S \setminus \Lambda$ is such that $u \le_S b$ for all $b \in \operatorname{Betti}(S) \setminus \Lambda$ with $b \le_S s$, then $\operatorname{B}(s; \Lambda)$ is connected in $\nabla_{s}$.
\end{lemma}
\begin{proof}
  Assume that $\Lambda=\{  b_1<_S\dots<_S b_l\}$, and so $u=b_l$. Write $\Lambda_i=\downarrow\! b_i=\{b_1<_S \dots<_S b_i\}$, for $i\in\{1,\ldots,l\}$. We proceed by induction on $l$, the size of $\Lambda$. Note that, by definition of $\downarrow\! u$, $b_1$ is Betti-minimal and the only Betti-minimal element with $b_1 \le_S u$. Let $s \in S \setminus \Lambda$ be such that $u \le_S b$ for every $b \in \operatorname{Betti}(S) \setminus \Lambda$ with $b \le_S s$. If $b \in \operatorname{Betti}(S)$ with $b \le_S s$, then either $b \in \Lambda$ and $b_1 \le_S b$, or $b \not \in \Lambda$ and $b_1 \le_S u \le_S b$. In any case, we have shown that $b_1 \le_S b$ for every $b \in \operatorname{Betti}(S)$ with $b \le_S s$. Hence, either $\mathfrak{d}(s) = 1$ or $b_1$ is the only Betti-minimal element with $b_1 \le_S s$. We will use this fact in our induction.
   
  First, we study the case $l=1$. Note that either $\mathfrak{d}(s) = 1$ or $b_1$ is the only minimal element of $S$ with $b_1 \le_S s$. In the first case, we have $\operatorname{B}(s; b_1) = \operatorname{Z}(s)$. In the second case, Lemma~\ref{lem:brf:existence}, with $\Lambda=\{b_1\}$, yields that $\operatorname{B}(s; b_1)$ is non-empty. Its connectivity follows from Lemma~\ref{lem:brf:1}.
 
  If $l\ge 2$, let us assume that the result holds for $l-1$. If $\operatorname{B}(s; \Lambda) = \operatorname{B}(s; \Lambda_{l-1})$, then we are done by the induction hypothesis. Let us consider the case $\operatorname{B}(s; \Lambda) \ne \operatorname{B}(s; \Lambda_{l-1})$. In this case we have $b_l \le_S s$, so $b_1$ is the only minimal Betti element of $S$ with $b_1 \le_S s$. There are two cases depending on the number of factorizations of $s-b_l$.
  
{\sc Case 1:}     $\mathfrak{d}(s-b_l) \ge 2$. Notice that, 
under this assumption, $b_1$ is the only minimal Betti element of $S$ with $b_1 \le_S s-b_l$. Let $z_1 \in \operatorname{B}(s; \Lambda) \setminus \operatorname{B}(s; \Lambda_{l-1})$. There is $y \in \operatorname{I}(b_l)$ such that $y < z_1$. Let $x \in \operatorname{I}(b_1)$. Since $b_1 \le_S s-b_l$, Lemma~\ref{lem:brf:existence} provides us with $z \in \operatorname{B}(s-b_l; \Lambda)$ such that $x < z$. Thus, we have $z_2 = z+y \in \operatorname{B}(s; \Lambda_l)$ and $z_1 \cdot z_2 \ne 0$. Moreover, there is $z_3 \in \operatorname{B}(s; b_1) \subseteq \operatorname{B}(s; \Lambda_{l-1})$ with $x < z_3$ (Lemma~\ref{lem:brf:1}) and, in particular, $z_2 \cdot z_3 \ne 0$.  From the arbitrary choice of $z_1$ and the fact that $\operatorname{B}(s; \Lambda_{l-1})$ is connected, it follows that $\operatorname{B}(s; \Lambda)$ is also connected.
      
{\sc Case 2:} $\mathfrak{d}(s- b_l) = 1$. Set $\omega_2 = s - b_l$, and let $w_2$ be the unique factorization of $\omega_2$. If $z \in \operatorname{B}(s; \Lambda) \setminus \operatorname{B}(s; \Lambda_{l-1})$, then there is $y \in \operatorname{I}(b_l)$ with $y < z$. Note that $z -y$ is a factorization of $\omega_2$, whence  $z = w_2 + y$. That is, we have shown that
\[ \operatorname{B}(s; \Lambda) \setminus \operatorname{B}(s; \Lambda_{l-1}) \subseteq \{w_2 + y \colon y \in \operatorname{I}(b_l) \}. \]
Let us suppose that  $\operatorname{B}(s; \Lambda)$ has at least two $R$-classes in order to obtain a contradiction. Since $\operatorname{B}(s; \Lambda_{l-1})$ and $\operatorname{B}(s; \Lambda) \setminus \operatorname{B}(s; \Lambda_{l-1})$ are connected, the only option is $z \cdot y = 0$ for every $z \in \operatorname{B}(s; \Lambda_{l-1})$ and $y \in \operatorname{B}(s; \Lambda) \setminus \operatorname{B}(s; \Lambda_{l-1})$.  Write $b_l = \omega_1 + q b_1$ with $\omega_1 \in \Ap(S; b_1)$ and $q\in \mathbb{N}$ (this implies $\operatorname{Z}(\omega_1) = \{w_1\}$ by Corollary~\ref{cor:isolated}). We have $s = \omega_1 + \omega_2 + q b_1$. There are two possible subcases, each of which yields a contradiction.
\par {\sc Subcase 2.1:} $\mathfrak{d}(\omega_1+\omega_2) = 1$. The factorization $z = w_1 + w_2 + q x$ is in $\operatorname{B}(s; \Lambda)$ for any $x \in \operatorname{I}(b_1)$, but it is not disjoint with any element of $\operatorname{B}(s; \Lambda) \setminus \operatorname{B}(s; \Lambda_{l-1})$, a contradiction.   
\par {\sc Subcase 2.2:} $\mathfrak{d}(\omega_1+\omega_2) \ge 2$. 
In light of Lemma~\ref{lem:isolated}, there exist $b \in \operatorname{Betti}(S)$ and $y \in \operatorname{I}(b)$ such that $y \le w_1 + w_2$. \ For all $x \in\operatorname{I}(\Lambda_{l-1})$ we have $\varphi(x) \le_S b_l \le_S s$ and there is $z \in \operatorname{B}(s; \Lambda_{l-1})$ with $x < z$ (Lemma~\ref{lem:brf:existence}). Under the standing assumption, $z \cdot w_2 = 0$. In particular, we have $w_2 \cdot x = 0$. It follows that the factorization $w_2$ is disjoint with any isolated factorization of the Betti elements $b_1, \ldots, b_{l-1}$. Since $y \le w_1 + w_2$, $y \in \operatorname{I}(b)$ and $w_1, w_2 \in \operatorname{I}_s(S)$, we have $w_1 \cdot y \ne 0$ and $w_2 \cdot y \ne 0$. Hence, it follows that $y$ is not an isolated factorization of any of the elements $b_1, \ldots, b_{l-1}$, so $b \not \in \Lambda_{l-1}$. Since $b \le _S \omega_1+\omega_2 \le_S s$, either $b \in \Lambda$ or $b_l \le_S b$ by hypothesis. We conclude that $b_l \le_S b \le_S \omega_1 + \omega_2$. Write $\omega_1+\omega_2 = \omega + p b_1 + b_l$ with $\omega \in \Ap(S; b_1)$ and $p \ge 0$. Since $s = \omega_2 + b_l$ and $ s = \omega_1 + \omega_2 + qb_1 = \omega + (p+q)b_1 + b_l$, we obtain $\omega_2 = s - b_l = \omega + (p+q)b_1$. Recall that $\mathfrak{d}(\omega_2) = 1$. This forces $p = 0 = q$, a contradiction because $b_l = \omega_1 + q b_1$ and $\mathfrak{d}(b_l) \ge 2$.
\end{proof}
We now have the ingredients necessary to prove Theorem~\ref{thm:brf:sorted}, which determines the number of isolated factorizations of any $b \in \operatorname{U}(\operatorname{Betti}(S))$.

\begin{theorem} \label{thm:brf:sorted}
  Let $S$ be a numerical semigroup and let  $b \in \operatorname{U}(\operatorname{Betti}(S))$. Then either $b$ is minimal and all its factorizations are isolated, or the number of isolated factorizations of $b$ equals its number of $R$-classes minus $1$. 
\end{theorem}
\begin{proof}
  Let $\downarrow\negthinspace b=\{b_1<_S \dots<_S b_l\}\subseteq \operatorname{Betti}(S)$. 
  If $l = 1$, then $b$ is Betti-minimal and its factorizations are isolated (see Proposition~\ref{prop:minimal-betti}). Otherwise, we apply Lemma~\ref{lem:brf:sorted} to $b_1, \ldots, b_l$, and conclude that $\operatorname{B}(b; \Lambda)$ is connected for $\Lambda = \{b_1, \ldots, b_{l-1}\}$. By Lemma~\ref{lem:brf:iso}, we find that $\operatorname{Z}(b) = \operatorname{B}(b; \Lambda) \cup \operatorname{I}(b),$ thus the number of isolated factorizations of $b$ is one less than the number of  $R$-classes.
\end{proof}

\begin{corollary} \label{cor:brf:b2}
  Let $S$ a numerical semigroup. Write $\operatorname{Betti}(S) = \{b_1 < b_2 < \cdots < b_k\}$. Let us assume that $k \ge 2$. 
  \begin{enumerate}[{\rm a)}]
  \item If $b_2 - b_1 \not\in S$, then $b_2$ is Betti-minimal; that is, all its factorizations are isolated.
  \item If $b_1 \le_S b_2$, then $b_2$ has $\operatorname{nc}(\nabla_{b_2})-1$ isolated factorizations.
  \end{enumerate}
\end{corollary}
\begin{proof}
  This is a direct consequence of Theorem~\ref{thm:brf:sorted}.
\end{proof}
Theorem~\ref{thm:brf:sorted} gives us some information about the smallest Betti elements. Let $b_1 = \min \operatorname{Betti}(S)$. It is clear that $b_1$ is Betti-minimal. Let us assume that there exists $b_2 = \min(\operatorname{Betti}(S) \setminus \{b_1\})$. Then either $b_2$ is Betti-minimal ($b_2 - b_1 \not \in S$), or $b_1 \le_S b_2$. In the latter case we can apply Theorem~\ref{thm:brf:sorted} to conclude that $\operatorname{i}(b_2) = \operatorname{nc}(\nabla_{b_2})-1$. Therefore, $b_2$ always has isolated factorizations. 

\begin{example} \label{ex:isolated}
Let $S=\langle 10,15,16,17,19\rangle$. Then the Hasse diagram of $(\operatorname{Betti}(S),\le_S)$ looks as follows:
 
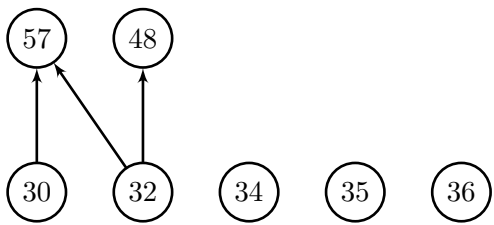
\begin{center}
\begin{tikzpicture}[>=latex',line join=bevel,]
  \pgfsetlinewidth{1bp}
\pgfsetcolor{black}
  \draw [<-] (51.0bp,57.817bp) .. controls (51.0bp,47.547bp) and (51.0bp,32.408bp)  .. (51.0bp,22.149bp);
  \draw [<-] (11.0bp,57.817bp) .. controls (11.0bp,47.547bp) and (11.0bp,32.408bp)  .. (11.0bp,22.149bp);
  \draw [<-] (17.25bp,59.938bp) .. controls (24.721bp,49.105bp) and (37.152bp,31.079bp)  .. (44.662bp,20.19bp);
\begin{scope}
  \definecolor{strokecol}{rgb}{0.0,0.0,0.0};
  \pgfsetstrokecolor{strokecol}
  \draw (51.0bp,69.0bp) ellipse (11.0bp and 11.0bp);
  \draw (51.0bp,69.0bp) node {48};
\end{scope}
\begin{scope}
  \definecolor{strokecol}{rgb}{0.0,0.0,0.0};
  \pgfsetstrokecolor{strokecol}
  \draw (51.0bp,11.0bp) ellipse (11.0bp and 11.0bp);
  \draw (51.0bp,11.0bp) node {32};
\end{scope}
\begin{scope}
  \definecolor{strokecol}{rgb}{0.0,0.0,0.0};
  \pgfsetstrokecolor{strokecol}
  \draw (11.0bp,69.0bp) ellipse (11.0bp and 11.0bp);
  \draw (11.0bp,69.0bp) node {57};
\end{scope}
\begin{scope}
  \definecolor{strokecol}{rgb}{0.0,0.0,0.0};
  \pgfsetstrokecolor{strokecol}
  \draw (11.0bp,11.0bp) ellipse (11.0bp and 11.0bp);
  \draw (11.0bp,11.0bp) node {30};
\end{scope}
\begin{scope}
  \definecolor{strokecol}{rgb}{0.0,0.0,0.0};
  \pgfsetstrokecolor{strokecol}
  \draw (171.0bp,11.0bp) ellipse (11.0bp and 11.0bp);
  \draw (171.0bp,11.0bp) node {36};
\end{scope}
\begin{scope}
  \definecolor{strokecol}{rgb}{0.0,0.0,0.0};
  \pgfsetstrokecolor{strokecol}
  \draw (131.0bp,11.0bp) ellipse (11.0bp and 11.0bp);
  \draw (131.0bp,11.0bp) node {35};
\end{scope}
\begin{scope}
  \definecolor{strokecol}{rgb}{0.0,0.0,0.0};
  \pgfsetstrokecolor{strokecol}
  \draw (91.0bp,11.0bp) ellipse (11.0bp and 11.0bp);
  \draw (91.0bp,11.0bp) node {34};
\end{scope}
\end{tikzpicture}
\end{center}
Hence the minimal elements of $\operatorname{Betti}(S)$ are $30, 32, 34, 35, 36$. The factorizations of these elements are all isolated. Theorem~\ref{thm:brf:sorted} allows us to conclude that $48$ has isolated factorizations and that $\operatorname{nc}(\nabla_{48}) = \operatorname{i}(48)+1$. Note that this result does not provide information about the factorizations of $57$ ($\downarrow\! 57=\{30,32,57\}$). In fact, one can check that $57$ has an isolated factorization with the help of the GAP package \texttt{numericalsgps}.
\end{example}

Next, we want to give an expression for $\operatorname{B}(s; \Lambda \cup \{b\})$ in terms of $\operatorname{B}(s -j b; \Lambda),$ for suitable $j,$ which is meant to be of the same shape as that given in recursion~\eqref{eq:ri:recurrence}.

\begin{lemma} \label{lem:brf:recurrence}
  Let $S$ be a numerical semigroup. Let $\Lambda \subseteq \operatorname{IBetti}(S)$ with $\Lambda \ne \emptyset$ and $b \in \operatorname{IBetti}(S) \setminus \Lambda$. Then for every $s \in S$ we have
  \[\operatorname{B}(s; \Lambda \cup \{b\}) = \bigcup_{j = 0}^{q_s} \left(\operatorname{B}(s- j b; \Lambda) + \sum_{i = 1}^j \operatorname{I}(b)\right),\]
  where $q_s$ is the largest integer such that $q_s b \le_S s$. In particular, 
  \[\lvert\operatorname{B}(s; \Lambda \cup \{b\})\rvert \le \sum_{j = 0}^{q_s} \lvert\operatorname{B}(s- j b; \Lambda)\rvert \binom{\operatorname{i}(b) + j -1}{j}.\]
\end{lemma}
\begin{proof}
  It is clear that $\bigcup_{j = 0}^{q_s} \left(\operatorname{B}(s- j b; \Lambda) + \sum_{i = 1}^j \operatorname{I}(b)\right) \subseteq \operatorname{B}(s; \Lambda \cup \{b\})$. Let $z \in \operatorname{B}(s; \Lambda \cup \{b\})$. If $z \not \in \operatorname{B}(s; \Lambda)$, then there is $y \in \operatorname{I}(b)$ such that $y \le z$ and $z-y \in \operatorname{B}(s-b; \Lambda\cup\{b\})$. We can repeat this argument a finite number of times until we find $y_1, \ldots, y_j \in \operatorname{I}(b)$ and $x \in \operatorname{B}(s-jb; \Lambda)$ such that $z = x + y_1 + \cdots + y_j$.  Hence, we obtain $z \in \operatorname{B}(s- j b; \Lambda) + \sum_{i = 1}^j \operatorname{I}(b)$.
\end{proof}

\begin{example}
  Let us consider the numerical semigroup $S = \langle 4, 5, 6\rangle$. We have $\operatorname{Betti}(S) = \{10, 12\}$, $\operatorname{Z}(10) = \{(1,0,1), (0,2,0)\}$ and $\operatorname{Z}(12) = \{(3,0,0), (0,0,2)\}$. Since $72 = 6 \cdot 12$, by Lemmas~\ref{lem:brf:iso} and \ref{lem:brf:recurrence}  we find that 
  \[\operatorname{Z}(72) = \operatorname{B}(72; \{10, 12\}) = \bigcup_{j = 0}^{6} \left(\operatorname{B}((6 - j) 12; 10) + \sum_{i = 1}^j \operatorname{I}(12)\right).\]
   Note that $(6, 0, 8) \in \operatorname{Z}(72)$ and
  \[(6, 0, 8) = 6(1,0,1) + (0,0,2) = 2 (3,0,0) + 4(0,0,2).\]
  This union is therefore not disjoint and the inequality given in Lemma~\ref{lem:brf:recurrence} can be strict.
\end{example}

The following lemma shows that, under the hypotheses of Theorem~\ref{thm:brf:sorted}, the upper bound given in Lemma~\ref{lem:brf:recurrence} can be attained. Note that this recurrent expression has already arisen in~\eqref{eq:ri:recurrence}.

\begin{lemma} \label{lem:brf:sorted:2}
  Let $S$ be a numerical semigroup. Let $u \in \operatorname{U}(\operatorname{Betti}(S))$ and let $\Lambda = \downarrow\! u$. Then, for every $s \in S$ and $z \in \operatorname{B}(s; \Lambda)$, there are unique $w \in \operatorname{I}_s(S)$ and $x_1, \ldots, x_t \in \operatorname{I}(\Lambda)$ such that $z = w + x_1 + \cdots + x_t$. Moreover, we have
  \[ \lvert \operatorname{B}(s; \Lambda) \rvert = \sum_{j = 0}^{q_s} \lvert\operatorname{B}(s- j u; \Lambda \setminus \{u\})\rvert \binom{\operatorname{i}(u) + j -1}{j}, \]
  where $q_s$ is the largest integer such that $q_s u \le_S s$.
\end{lemma}
\begin{proof}
   By Theorem~\ref{thm:brf:sorted}, we have $\Lambda \subseteq \operatorname{IBetti}(S)$. Let $s \in S$. If $\operatorname{B}(s; \Lambda) = \emptyset$, then we are done. Let us assume that $\operatorname{B}(s; \Lambda) \ne \emptyset$ and let $z \in \operatorname{B}(s; \Lambda)$. The definition of $\operatorname{B}(s; \Lambda)$ ensures the existence of $w \in \operatorname{I}_s(S)$ and $x_1, \ldots, x_t \in \operatorname{I}(\Lambda)$ such that $z = w + x_1 + \cdots + x_t$. We show that this expression is unique. Let $w_1, w_2 \in \operatorname{I}_s(S)$ and, for each $x \in \operatorname{I}(\Lambda)$, let $p_x$ and $q_x$ be non-negative integers such that
  \[ z = w_1 + \sum_{x \in \operatorname{I}(\Lambda)} p_x x = w_2 + \sum_{x \in \operatorname{I}(\Lambda)} q_x x . \]
  In light of Lemma~\ref{lem:disjoint-betti}, the supports of the elements of $\operatorname{I}(\Lambda)$ are disjoint. If there is $x \in \operatorname{I}(\Lambda)$ such that $p_x \ne q_x$, then either $x < w_1$, or $x < w_2$, contradicting the fact that $w_1, w_2 \in \operatorname{I}_s(S)$, see Lemma~\ref{lem:isolated}. Therefore we have $p_x = q_x$ for every $x \in \Lambda$ and $w_1 = w_2$. 
  
  As a consequence, the union given in Lemma~\ref{lem:brf:recurrence} is disjoint. Moreover, we have 
  \[ \left\lvert\operatorname{B}(s- j u; \Lambda \setminus \{u\}) + \sum\nolimits_{i = 1}^j \operatorname{I}(u)\right\rvert = \left\lvert\operatorname{B}(s- j u; \Lambda \setminus \{u\})\right\rvert \binom{\operatorname{i}(u) + j -1}{j} \]
  and the result follows.
\end{proof}

\subsection{Completing the proof of Theorem~\ref{thm:ces:betti}} \label{sec:thm2:thm2}

We now have all the ingredients necessary to complete the proof of Theorem~\ref{thm:ces:betti}.
    
\begin{proof}[Proof of Theorem \ref{thm:ces:betti}]
First, we prove that if $\eta\in \operatorname{U}( \mathcal{E}(S))$, then $\eta \in \operatorname{U}(\operatorname{Betti}(S))$ and $e_\eta$ is as in the statement. The tools developed in this part of the proof will be helpful when proving the other inclusion. As we follow the proof idea explained in Section~\ref{sec:thm2:proof-outline}, we recommend the reader to have a look at Section~\ref{sec:thm2:proof-outline} before reading this proof. Let $\Lambda= \{d \in \mathcal{E}(S) : d \le_S \eta \}$. Since $\eta \in \operatorname{U}( \mathcal{E}(S))$, we can write $\Lambda=\{b_1<_S \dots<_S b_l\}$ and $b_l = \eta$. For each $i \in \{1, \ldots, l\}$, define $\Lambda_i = \{b_1, \ldots, b_i\}$ and $r_{i-1}(s)$ as in~\eqref{eq:ri:1}, identity which we recall below for the convenience of the reader:
\begin{equation*} 
    \sum_{s \in S} r_{i-1}(s)x^s = \operatorname{H}_S(x) \prod_{j = 1}^{i-1} (1 - x^{b_j})^{-e_{b_j}}.
\end{equation*}
Note that from our hypothesis it follows that $b_1$ is minimal in $\mathcal{E}(S)$. By Theorem~\ref{thm:betti-minimals:omega}, $b_1$ is Betti-minimal. We prove by induction on $i \in \{1, \ldots, l\}$ the following assertions:
  \begin{enumerate}[{\rm (a)}]
      \item \label{item:cyclo:lambda:no-more-betti} if $b \le_S b_i$ for some $b \in \operatorname{Betti}(S)$, then $b \in \Lambda_i$;
      \item \label{item:cyclo:lambda:betti} $b_i \in \operatorname{U}(\operatorname{Betti}(S))$ and $e_{b_i} > 0$;
      \item \label{item:cyclo:lambda:r} $r_i(s) = \left|\operatorname{B}(s; \Lambda_i)\right|$ for every $s \in S$ such that $b_1$ is the only Betti-minimal element with $b_1 \le_S s.$
   \end{enumerate}
  First, we study the case $i =1$. As a consequence of Theorem~\ref{thm:betti-minimals:omega}, $b_1$ Betti-minimal and $\operatorname{i}(b_1) = \mathfrak{d}(b_1) = e_{b_1}+1 \ge 2$.
  In particular,~\eqref{item:cyclo:lambda:no-more-betti} and~\eqref{item:cyclo:lambda:betti} 
  hold for $i = 1$. Note that, by~\eqref{eq:pol:apery}, we have $\operatorname{H}_S(x) = \sum_{\omega \in \Ap(S; b_1)} x^\omega \sum_{j = 0}^\infty x^{j b_1}$. In conjunction with~\eqref{eq:taylor} and $\operatorname{i}(b_1) = e_{b_1}+1$, this yields
\begin{equation} 
  \begin{aligned}\label{eq:rb}
    \sum_{s = 0}^\infty r_1(s) x^s & = \sum_{\omega \in \Ap(S; b_1)} x^\omega \bigg(\sum_{j = 0}^\infty x^{j b_1}\bigg)^{\operatorname{i}(b_1)}\\
                                   & = \sum_{\omega \in \Ap(S; b_1)} x^\omega \sum_{j = 0}^\infty \binom{\operatorname{i}(b_1)+j-1}{j} x^{j b_1} \\ 
                                   & = \sum_{s \in S} \binom{\operatorname{i}(b_1) + q_s - 1}{q_s} x^{s},
\end{aligned}
\end{equation}
where $q_s$ is the unique non-negative integer such that $s - q_s b_1 \in \Ap(S; b_1)$. 
Let $s \in S$ be such that $b_1$ is the only Betti-minimal element with $b_1 \le_S s$. Then from Corollary~\ref{cor:isolated} it follows that $\omega_s = s - q_s b_1$ has only one factorization. By combining Corollary~\ref{cor:brf:1} and~\eqref{eq:rb}, we conclude that $r_1(s) = \left|\operatorname{B}(s; \Lambda_1)\right|,$ as desired.
  
  Now assume that~\eqref{item:cyclo:lambda:no-more-betti},~\eqref{item:cyclo:lambda:betti} and~\eqref{item:cyclo:lambda:r} hold for $i-1 \in \{1, \ldots, l-1\}$ and let us prove that they also hold for $i$.  We prove each of the induction hypotheses for $i$ separately. In doing so, we will use the identities \eqref{eq:ri:values} several times, which, for the sake of readability, we also recall here:
  \begin{equation*} 
      \begin{aligned}
        r_{i-1}(s) &= \mathfrak{d}(s)  \qquad \qquad \, \text{when } \{d \in \mathcal{E}(S) : d \le_S s\} \subseteq \Lambda_{i-1}, \\
        r_{i-1}(s) &= \mathfrak{d}(s) - e_s \qquad \text{when } s \in \operatorname{Minimals}_{\le_S} \left(\mathcal{E}(S) \setminus \Lambda_{i-1}\right).
        \end{aligned}
  \end{equation*}
  
\noindent (a) We proceed by deriving a contradiction. Let us assume that there is $b \in \operatorname{Betti}(S) \setminus \Lambda_{i-1}$ such that $b <_S b_i$. Then there exists an element
    \[ \beta \in \operatorname{Minimals}_{\le_S} \{b \in \operatorname{Betti}(S) \setminus \Lambda_{i-1} : b <_S b_i\}.  \]
    Let $D = \{ d \in \operatorname{Betti}(S) : d <_S \beta \}$. Since $\beta <_S b_i$, from minimality in the choice of $\beta$ it follows  that $D \subseteq \Lambda_{i-1}$. If $D = \emptyset$, then $\beta$ is Betti-minimal and, by Theorem~\ref{thm:betti-minimals:omega}, $\beta \in \mathcal{E}(S)$. We thus obtain $\beta \in \{d \in \mathcal{E}(S) : d <_S b_i\} = \Lambda_{i-1}$, but $\beta \not \in \Lambda_{i-1}$ by definition, a contradiction. We conclude that $\emptyset \ne D \subseteq \Lambda_{i-1}$. Hence, we have $b_1 \in d$, so $b_1 \le_S \beta <_S b_i$ and $b_1$ is the only Betti-minimal element that satisfies $b_1 \le_S \beta$. From our induction hypothesis, we obtain $\left|\operatorname{B}(\beta; \Lambda_{i-1})\right| = r_{i-1}(\beta)$. Note that $\{d \in \mathcal{E}(S) : d  \le_S \beta\} \subseteq \{d \in \mathcal{E}(S) : d  <_S b_i\} = \Lambda_{i-1}$. Hence, by~\eqref{eq:ri:values}, we find that $r_{i-1}(\beta) = \mathfrak{d}(\beta)$, so $\left|\operatorname{B}(\beta; \Lambda_{i-1})\right| = \mathfrak{d}(\beta)$. We can apply Lemma~\ref{lem:brf:sorted} with $u = b_{i-1}$ and $s = \beta$, finding that $\operatorname{B}(\beta; \Lambda_{i-1})$ is connected in $\nabla_\beta$. But we have shown that $\left|\operatorname{B}(b; \Lambda_{i-1})\right| = \mathfrak{d}(b)$ or, equivalently, $\operatorname{Z}(\beta) = \operatorname{B}(\beta; \Lambda_{i-1})$. This contradicts the fact that $\beta \in \operatorname{Betti}(S)$.\\

\noindent    (b) From Lemma~\ref{lem:brf:iso} and~\eqref{item:cyclo:lambda:no-more-betti} it follows that $\operatorname{Z}(b_i) \setminus \operatorname{I}_b(b_i)$, so $\left|\operatorname{B}(b_i; \Lambda_{i-1})\right| = \mathfrak{d}(b_i) - \operatorname{i}(b_i)$. Moreover, by our hypothesis we have $\left|\operatorname{B}(b_i; \Lambda_{i-1})\right| = r_{i-1}(b_i)$. Note that $b_i \in \operatorname{Minimals}_{\le_S} (\mathcal{E}(S) \setminus \Lambda_{i-1})$ by definition of $\Lambda_{i-1}$. Hence, by~\eqref{eq:ri:values} we obtain $r_{i-1}(b_i) = \mathfrak{d}(b_i) - e_{b_i}$.  We conclude that $\mathfrak{d}(b_i) - \operatorname{i}(b_i) = \mathfrak{d}(b_i) - e_{b_i}$, that is, $0 \le \operatorname{i}(b_i) = e_{b_i}$. Since $e_{b_i} \ne 0$, we have $\operatorname{i}(b_i) \ge 1$ and $b_i \in \operatorname{Betti}(S)$ because $b_i$ has at least two factorizations (Corollary~\ref{cor:isolated}). In view of~\eqref{item:cyclo:lambda:no-more-betti}, we have $\{b \in \operatorname{Betti}(S) : b \le_S b_i\} = \Lambda_i$, which 
by hypothesis is totally ordered, so $b \in \operatorname{U}(\operatorname{Betti}(S))$.\\
    
\noindent    (c)  Let $s \in S$ such that $b_1$ is the only Betti-minimal element with $b_1 \le_S s$. Note that thanks to 
(b) we can apply Lemma~\ref{lem:brf:sorted:2} with $u = b_i$. Recall that in \eqref{eq:ri:recurrence} we showed that
  \begin{equation*}
      r_i(s) = \sum_{j = 0}^{q_s} r_{i-1}(s - j b_i) \binom{\operatorname{i}(b_i) + j - 1}{j},
  \end{equation*}
    where $q_s$ is the largest integer such that $s - q_s b_i \in S$. This equation in combination with Lemma~\ref{lem:brf:sorted:2} yields
    \[ r_i(s) = \sum_{j = 0}^{q_s} \left|\operatorname{B}(s - j b_i; \Lambda_{i-1})\right| \binom{\operatorname{i}(b_i) + j - 1}{j} = \left|\operatorname{B}(s; \Lambda_i)\right|, \]
    which finishes the proof by induction.
  
  In the induction we have also shown that $e_{b_i} = \operatorname{i}(b_i)$, see the proof of our hypothesis~\eqref{item:cyclo:lambda:betti}. Since $b_i\in \operatorname{U}(\operatorname{Betti}(S))$, by Theorem~\ref{thm:brf:sorted} we have $\operatorname{nc}(\nabla_{b_i}) = \operatorname{i}(b_i) + 1 = e_{b_i} + 1$. The fact that $e_{b_1} = \operatorname{nc}(\nabla_{b_1})-1$ has been established in Theorem~\ref{thm:betti-minimals:omega}.

  Finally we show that $\operatorname{U}(\operatorname{Betti}(S)) \subseteq \operatorname{U}(\mathcal{E}(S))$. Let $u \in \operatorname{U}(\operatorname{Betti}(S))$. Let us write $\downarrow\! u = \{b \in \operatorname{Betti}(S) : b \le_S u\} = \{b_1 <_S \cdots <_S b_l = u\}$. We prove by induction on $i$ that $b_i \in \operatorname{U}(\mathcal{E}(S))$. Note that $b_1$ is Betti-minimal and, thus, $b_1 \in \operatorname{U}(\mathcal{E}(S))$ by Theorem~\ref{thm:betti-minimals:omega}. Let us assume that $b_1, \ldots, b_{i-1} \in \operatorname{U}(\mathcal{E}(S))$ and let us prove that $b_i \in \operatorname{U}(\mathcal{E}(S))$. Let $\Lambda_{i-1} = \{b_1, \ldots, b_{i-1}\}$. We consider $r_{i-1}(s)$ as in~\eqref{eq:ri:1}. In light of the induction hypothesis~\eqref{item:cyclo:lambda:r}, we have $r_{i-1}(s) = \left|\operatorname{B}(s, \Lambda_{i-1})\right|$ for every $s \in S$ such that $b_1$ is the only Betti-minimal element with $b_1 \le_S s$. In particular, we have $r_{i-1}(b_i) = \left|\operatorname{B}(b_i, \Lambda_{i-1})\right|$ and, thus, $r_{i-1}(b_i) = \mathfrak{d}(b_i) - \operatorname{i}(b_i)$, where we used Lemma~\ref{lem:brf:iso}. From Theorem~\ref{thm:brf:sorted}, we find that $\operatorname{i}(b_i) > 0$. Therefore, $r_{i-1}(b_i) < \mathfrak{d}(b_i)$ and, by~\eqref{eq:ri:2}, there exists $d \in \mathcal{E}(S) \setminus \Lambda_{i-1}$ with $d \le_S b_i$. Hence, there is $\alpha \in \operatorname{Minimals}_{\le_S} \{d \in \mathcal{E}(S) \setminus \Lambda_{i-1}\}$ with $\alpha \le_S b_i$. By~\eqref{eq:ri:values} we have $r_{i-1}(\alpha) = \mathfrak{d}(\alpha) - e_\alpha$. From the induction hypothesis~\eqref{item:cyclo:lambda:r}, we obtain $r_{i-1}(\alpha) = \left|\operatorname{B}(\alpha, \Lambda_{i-1})\right|$. Since $0 < r_1(s) \le r_{i-1}(s)$ by \eqref{eq:rb} and \eqref{eq:ri:recurrence}, we have $\operatorname{B}(\alpha, \Lambda_{i-1}) \ne \emptyset$. In view of Lemma~\ref{lem:brf:iso}, $\operatorname{B}(\alpha; \Lambda_{i-1}) = \operatorname{Z}(\alpha) \setminus \operatorname{I}_b(\alpha)$, so $1 \le r_{i-1}(\alpha) = \mathfrak{d}(\alpha) - \operatorname{i}(\alpha)$. We find that $\operatorname{i}(\alpha) = e_\alpha \ne 0$. We have $0 \ne e_\alpha = \operatorname{i}(\alpha)$. We conclude that $\alpha$ is a Betti element with $\alpha \le_S b_i$. Since $\alpha \not \in \Lambda_{i-1}$ by definition, we must have $ b_i = \alpha \in \mathcal{E}(S)$ and $\{d \in \mathcal{E}(S):  d \le_S b_i\} = \{b_1, \ldots, b_i\}$. We obtain 
  $b_i \in \operatorname{U}(\mathcal{E}(S)),$ as wanted.
\end{proof}

\begin{example} \label{ex:ces:betti}
Here we can see Theorem~\ref{thm:ces:betti} in 
action for a couple of numerical semigroups.
\begin{enumerate}[{\rm a)}]
    \item We consider again the semigroup from Example~\ref{ex:isolated}. Let $S=\langle 10,15,16,17,19\rangle$. Recall that the Hasse diagram of $(\operatorname{Betti}(S),\le_S)$ looks as follows:
 
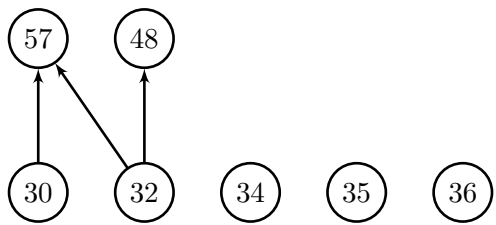
\begin{center}
\begin{tikzpicture}[>=latex',line join=bevel,]
  \pgfsetlinewidth{1bp}
\pgfsetcolor{black}
  \draw [<-] (51.0bp,57.817bp) .. controls (51.0bp,47.547bp) and (51.0bp,32.408bp)  .. (51.0bp,22.149bp);
  \draw [<-] (11.0bp,57.817bp) .. controls (11.0bp,47.547bp) and (11.0bp,32.408bp)  .. (11.0bp,22.149bp);
  \draw [<-] (17.25bp,59.938bp) .. controls (24.721bp,49.105bp) and (37.152bp,31.079bp)  .. (44.662bp,20.19bp);
\begin{scope}
  \definecolor{strokecol}{rgb}{0.0,0.0,0.0};
  \pgfsetstrokecolor{strokecol}
  \draw (51.0bp,69.0bp) ellipse (11.0bp and 11.0bp);
  \draw (51.0bp,69.0bp) node {48};
\end{scope}
\begin{scope}
  \definecolor{strokecol}{rgb}{0.0,0.0,0.0};
  \pgfsetstrokecolor{strokecol}
  \draw (51.0bp,11.0bp) ellipse (11.0bp and 11.0bp);
  \draw (51.0bp,11.0bp) node {32};
\end{scope}
\begin{scope}
  \definecolor{strokecol}{rgb}{0.0,0.0,0.0};
  \pgfsetstrokecolor{strokecol}
  \draw (11.0bp,69.0bp) ellipse (11.0bp and 11.0bp);
  \draw (11.0bp,69.0bp) node {57};
\end{scope}
\begin{scope}
  \definecolor{strokecol}{rgb}{0.0,0.0,0.0};
  \pgfsetstrokecolor{strokecol}
  \draw (11.0bp,11.0bp) ellipse (11.0bp and 11.0bp);
  \draw (11.0bp,11.0bp) node {30};
\end{scope}
\begin{scope}
  \definecolor{strokecol}{rgb}{0.0,0.0,0.0};
  \pgfsetstrokecolor{strokecol}
  \draw (171.0bp,11.0bp) ellipse (11.0bp and 11.0bp);
  \draw (171.0bp,11.0bp) node {36};
\end{scope}
\begin{scope}
  \definecolor{strokecol}{rgb}{0.0,0.0,0.0};
  \pgfsetstrokecolor{strokecol}
  \draw (131.0bp,11.0bp) ellipse (11.0bp and 11.0bp);
  \draw (131.0bp,11.0bp) node {35};
\end{scope}
\begin{scope}
  \definecolor{strokecol}{rgb}{0.0,0.0,0.0};
  \pgfsetstrokecolor{strokecol}
  \draw (91.0bp,11.0bp) ellipse (11.0bp and 11.0bp);
  \draw (91.0bp,11.0bp) node {34};
\end{scope}
\end{tikzpicture}
\end{center}
Hence the minimal elements of $\operatorname{Betti}(S)$ are $30, 32, 34, 35$ and $36$. The set $\operatorname{U}(\operatorname{Betti}(S))$ consists of these Betti-minimal elements and the Betti element $48$. Therefore, by Theorem~\ref{thm:ces:betti}, we conclude that $\operatorname{U}(\mathcal{E}(S)) = \{30, 32, 34, 35, 36, 48\}$ and we can determine the exponents $e_s$ of these elements from their number of isolated factorizations.

\item An interesting application is finding Betti elements from $\mathcal{E}(S)$. Let us consider the semigroup $S=\langle 3,5,7\rangle$ of Example~\ref{ex:ces}\ref{item:ex:b}. We gave the first entries of the cyclotomic exponent sequence of $S$. The smallest elements of $\mathcal{E}(S)$ are $10,12,14,17, 19, \ldots.$ Note that $\operatorname{U}(\mathcal{E}(S)) = \{10, 12, 14\}$ since any other element in $\mathcal{E}(S)$ can be written as $\alpha + 3 j$ for some $\alpha \in \{10, 12, 14\}$ and $j \ge 0$. Therefore, we have $\operatorname{U}(\operatorname{Betti}(S)) = \{10, 12, 14\} = \operatorname{Minimals}_{\le_S} \operatorname{Betti}(S)$, and each one of these elements has only two factorizations (their cyclotomic exponents are $1$).
\end{enumerate}
\end{example}

\section{Betti-sorted and Betti-divisible numerical semigroups} \label{sec:thm3}

  In this section we prove Theorem~\ref{thm:ces:char}, which characterizes Betti-sorted and Betti-divisible numerical semigroups in terms of their cyclotomic exponent sequences. Recall that $S$ is Betti-sorted if  $\operatorname{Betti}(S)$ is totally ordered by $\le_S,$ and that $S$ is Betti-divisible if  $\operatorname{Betti}(S)$ is totally ordered by the divisibility order in $\mathbb{N}$. Our characterizations are consequences of Theorem~\ref{thm:ces:betti}. We will use the following result on ordered sets.
   
  \begin{lemma} \label{lem:ordered-sets}
    Let $(X,\le)$ be an ordered set. Then $X$ is totally ordered if and only if $\operatorname{U}(X)$ is totally ordered. 
  \end{lemma}
  \begin{proof}
    First, if $X$ is totally ordered, then any subset of $X$, and in particular $\operatorname{U}(X)$, is totally ordered. Now let us assume that $\operatorname{U}(X)$ is totally ordered. Suppose $\operatorname{U}(X) \ne X$ in order to obtain a contradiction. Then we can choose $\alpha \in \operatorname{Minimals}_{\le} \left( X \setminus \operatorname{U}(X)\right)$. We have $\{a \in X: a < \alpha\} \subseteq \operatorname{U}(X)$ by minimality of $\alpha$. Thus $\downarrow\! \alpha$ is of the form $\{a_1 < \cdots < a_k < \alpha\}$ for some $k \ge 0$ and $a_1, \ldots, a_k  \in \operatorname{U}(X)$. We conclude that $\alpha \in \operatorname{U}(X)$, a contradiction. Therefore, $X = \operatorname{U}(X)$ and $X$ is totally ordered.
  \end{proof}
   
  \begin{lemma} \label{lem:betti-sorted}
    Let $S$ be a numerical semigroup. Then $S$ is Betti-sorted if and only if  $\mathcal{E}(S)$ is totally ordered by $\le_S$. Moreover, if this is the case, then $\operatorname{Betti}(S) = \mathcal{E}(S)$.
  \end{lemma}
  \begin{proof}
   In view of Theorem~\ref{thm:ces:betti} and Lemma~\ref{lem:ordered-sets}, $S$ is Betti-sorted if and only if $\operatorname{U}(\operatorname{Betti}(S)) = \operatorname{U}(\mathcal{E}(S))$ is totally ordered by $\le_S$ or, equivalently, $\mathcal{E}(S)$ is totally ordered by $\le_S$.
\end{proof}

This gives the following alternative proof of the fact that Betti-sorted numerical semigroups are complete intersections. For the original proof we refer to \cite{isolated}, where in fact the authors show the stronger result  that Betti-sorted numerical semigroups are free.

\begin{corollary} \label{cor:betti-sorted:ci}
  If $S$ is a Betti-sorted numerical semigroup, then $S$ is a complete intersection.
\end{corollary}
\begin{proof}
 In view of Lemma~\ref{lem:betti-sorted}, we have $\mathcal{E}(S) = \mathrm{U}(\mathcal{E}(S))$. By applying Theorem~\ref{thm:ces:betti} we find that $e_b = \operatorname{nc}(\nabla_b)-1$ for every $b \in \operatorname{Betti}(S) = \mathcal{E}(S)$. Let $A$ be the minimal system of generators of $S$. With the help of Theorem~\ref{thm:ces:generators}, we conclude that
\begin{equation*}
    \operatorname{H}_S(x)=\frac{\prod_{b\in\operatorname{Betti}(S)}(1-x^b)^{\operatorname{nc}(\nabla_b)-1}}
{\prod_{n \in A}(1-x^{n})}.
\end{equation*}
Therefore, $S$ is a complete intersection by Proposition~\ref{prop:ci-char}.
\end{proof}

\begin{lemma}
  \label{lem:betti-divisible}
    Let $S$ be a numerical semigroup. Then $S$ is Betti-divisible if and only if $\mathcal{E}(S)$ is totally ordered by the divisibility order.
\end{lemma}
\begin{proof}
  Let $S$ be a numerical semigroup such that either $\operatorname{Betti}(S)$ or $\mathcal{E}(S)$ is totally ordered by the divisibility order. Then, by Lemma~\ref{lem:betti-sorted}, $S$ is Betti-sorted and $\operatorname{Betti}(S) = \mathcal{E}(S)$. It follows that $\operatorname{Betti}(S)$ and $\mathcal{E}(S)$ are totally ordered by the divisibility order. 
\end{proof}

Betti-divisible numerical semigroups are rare, but they have a very rich structure. In fact, it can be shown that these are the numerical semigroups that are free for any arrangement of their minimal generators, see \cite[Theorem 7.10]{isolated}. 

\begin{lemma} \label{lem:unique-betti}
  Let $S$ be a numerical semigroup minimally generated by $A$. Then $S$ has a unique Betti element if and only if $\mathcal{E}(S)$ is a singleton.
\end{lemma}
\begin{proof}
  Let $S$ be a numerical semigroup such that $\operatorname{Betti}(S)$ or $\mathcal{E}(S)$ is a singleton. Then $S$ is Betti-sorted by Lemma~\ref{lem:betti-sorted} and $\operatorname{Betti}(S) = \mathcal{E}(S)$, so both $\operatorname{Betti}(S)$ and $\mathcal{E}(S)$ are singletons.
\end{proof}

Theorem~\ref{thm:ces:char} now follows by combining Lemmas~\ref{lem:betti-sorted},~\ref{lem:betti-divisible} and~\ref{lem:unique-betti}. 

\section{Applications to cyclotomic numerical semigroups and open questions} \label{sec:cns}
 
We can now use our freshly enriched insight on the connections between cyclotomic exponent sequences and Betti elements to prove that certain cyclotomic numerical semigroups are complete intersections. We do so by showing that these numerical semigroups satisfy the hypotheses of Proposition~\ref{prop:ci-char} and are, as such, complete intersections. This approach has already been carried out in Corollary~\ref{cor:betti-sorted:ci}, where we showed that Betti-sorted numerical semigroups are complete intersections. In fact, here we extend Corollary~\ref{cor:betti-sorted:ci} to a larger family of numerical semigroups. First, let us consider the following conjectures.

\begin{conjecture} \label{con:msg}
	Let $S$ be a cyclotomic numerical semigroup and let $\mathbf{e}$ be its cyclotomic exponent sequence. Then $n \in \mathbb{N}$ is a minimal generator of $S$ if and only if $e_n < 0$.
\end{conjecture}
\begin{conjecture} \label{con:betti}
		Let $S$ be a cyclotomic numerical semigroup and let $\mathbf{e}$ be its cyclotomic exponent sequence. Then $e_b = \operatorname{nc}(\nabla_b) -1$ for all $b \in \operatorname{Betti}(S)$. In particular, we have  $\operatorname{Betti}(S) \subseteq \mathcal{E}(S)$.
\end{conjecture}

These conjectures are motivated by the following result.

\begin{proposition} \label{prop:conjectures}
Conjecture~\ref{con:cyclo-ci} holds if and only if Conjectures~\ref{con:msg} and~\ref{con:betti} hold.
\end{proposition}
\begin{proof}
First, we note that Conjectures \ref{con:msg} and \ref{con:betti} are directly implied by Conjecture~\ref{con:cyclo-ci} and Proposition~\ref{prop:ci-char}. Now let us assume that $S$ is a cyclotomic numerical semigroup such that Conjectures~\ref{con:msg} and~\ref{con:betti} hold for $S$ and let us prove that Conjecture~\ref{con:cyclo-ci} holds for $S$ or, equivalently, 
  that $S$ is a complete intersection. 
  Since Conjecture~\ref{con:msg} holds for $S$, from Proposition~\ref{prop:sum-e} and Theorem~\ref{thm:ces:generators} we obtain 
  \begin{equation*}
      0 = \sum_{d\ge1 } e_d = -\operatorname{e}(S) + \sum_{\substack{d \ge1\\e_d > 0}} e_d.
  \end{equation*}
  From these equalities, Conjecture~\ref{con:betti} and the fact that $e_1 = 1$ (by Theorem~\ref{thm:ces:generators}), we conclude that
   \[
      \operatorname{e}(S)  = \sum_{\substack{d \ge1\\e_d > 0}} e_d  \ge 1 + \! \! \! \! \! \!  \! \! \! \! \! \! \sum_{\qquad b\in\operatorname{Betti}(S)} (\operatorname{nc}(\nabla_b)-1),
  \]
  which shows that the cardinality of any minimal presentation of $S$ is bounded by $\operatorname{e}(S) -1,$ and therefore that $S$ is a complete intersection (see Section~\ref{sec:pre:ci}).
\end{proof}

Theorem~\ref{thm:ces:generators} shows one direction of Conjecture~\ref{con:msg}. Here we show that this conjecture holds for a large set of cyclotomic numerical semigroups.

\begin{corollary} \label{cor:cyclo:forest:1}
  Let $S$ be a numerical semigroup minimally generated by $A$. If $\operatorname{U}(\mathcal{E}(S)) = \mathcal{E}(S)$, then $\mathcal{E}(S) \subseteq \operatorname{Betti}(S)$, $S$ is cyclotomic and Conjecture~\ref{con:msg} holds for $S$.
\end{corollary}
\begin{proof}
  From Theorem~\ref{thm:ces:betti}, we find that $\mathcal{E}(S) = \operatorname{U}(\operatorname{Betti}(S))$ and that $e_b = \operatorname{nc}(\nabla_b) - 1 > 0$ for every $b \in \mathcal{E}(S)$. In particular, $\mathcal{E}(S)$ is finite and, thus, $S$ is cyclotomic (Definition~\ref{def:CNS}). 
  Let $n \in \mathbb{N}$ with $e_n < 0$. We have $n \not \in \mathcal{E}(S)$ because $e_b > 0$ for every $b \in \mathcal{E}(S)$. Moreover, recall that $e_0 = 0$ and $e_1 = 1$ by Theorem~\ref{thm:ces:generators}, so $n \ge 2$. Since $\mathcal{E}(S)$ is the set of positive integers $j$ such that $j \ge 2$, $e_j \ne 0$ and $j$ is not a minimal generator, we conclude that $n$ is a minimal generator of $S$.
\end{proof}

As already mentioned in the introduction, computations suggest that these numerical semigroups arise very frequently. 

\begin{corollary} \label{cor:cyclo:forest:2}
  Let $S$ be a cyclotomic numerical semigroup. If $\operatorname{Betti}(S) = \operatorname{U}(\operatorname{Betti}(S))$ and Conjecture~\ref{con:msg} holds for $S$, then $S$ is a complete intersection.
\end{corollary}
\begin{proof}
  From Theorem~\ref{thm:ces:betti} and $\operatorname{Betti}(S) = \operatorname{U}(\operatorname{Betti}(S))$, we obtain $\operatorname{Betti}(S) \subseteq \mathcal{E}(S)$. The result now follows from Proposition~\ref{prop:conjectures}.
\end{proof}

\begin{corollary} \label{cor:cyclo:forest:3}
   Let $S$ be a numerical semigroup. If $\operatorname{Betti}(S) = \operatorname{U}(\operatorname{Betti}(S))$ and $\mathcal{E}(S) = \operatorname{U}(\mathcal{E}(S))$, then $S$ is a complete intersection.
\end{corollary}

\begin{proof}
 This follows by combining Corollaries~\ref{cor:cyclo:forest:1} and~\ref{cor:cyclo:forest:2}.
\end{proof}

\begin{example}
Let $S=\langle 8, 12, 18, 25\rangle$. Then \[\operatorname{P}_S(x)=\frac{(1-x)(1-x^{24})(1-x^{36})(1-x^{50})}{(1-x^8)(1-x^{12})(1-x^{18})(1-x^{25})}.\] Then $\mathcal{E}(S) = \operatorname{Betti}(S)=\{24,36,50\}$. The graph $(\operatorname{Betti}(S),\le_S)$ is depicted below. 

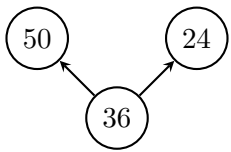
\begin{center}
\begin{tikzpicture}[
            > = stealth, 
            auto,
            node distance = 1.5cm, 
            thick 
        ]

        \tikzstyle{every state}=[
            draw = black,
            thick,
            fill = white,
            minimum size = 4mm
        ]

        \node[state] (s) {$36$};
        \node[state] (v1) [above left of=s] {$50$};
        \node[state] (v2) [above right of=s] {$24$};
        
        \path[->] (s) edge node {} (v1);
        \path[->] (s) edge node {} (v2);
    \end{tikzpicture}
\end{center}

 From this graph it follows that $\operatorname{Betti}(S) = \operatorname{U}(\operatorname{Betti}(S))$ and, thus, $S$ is a complete intersection. 
\end{example}

Finally, let us make a few comments on Conjecture~\ref{con:msg}. In \cite[Lemma 14]{cyclotomic} it is shown that this conjecture holds true under several restrictions on $S$. We notice that the restrictions in part (a) and (b) of Lemma 14 from \cite{cyclotomic} cannot both hold at the same time, hence the statement of \cite[Lemma 14]{cyclotomic} is void, in the sense that it does not find cyclotomic numerical semigroups satisfying Conjecture~\ref{con:msg}. We conclude this section by improving  \cite[Lemma 14]{cyclotomic}. 

\begin{proposition}
  Let $S$ be a cyclotomic numerical semigroup with cyclotomic exponent sequence $\bf{e}$. Let $j \in \mathbb{N}$ with $e_j < 0$ and $j < \min \{d \in \mathbb{N}: e_{d} > 0\}$. Then $j$ is a minimal generator of $S$. As a consequence, if $\max \{d \in \mathbb{N}: e_{d} < 0\} < \min \{d \in \mathbb{N}: e_{d} > 0\}$, then Conjecture~\ref{con:msg} holds for $S$.
\end{proposition}
\begin{proof}
  Let $j \in \mathbb{N}$ with $e_j < 0$ and $j < \min \{d \in \mathbb{N}: e_{d} > 0\}$. In view of Theorem~\ref{thm:ces:generators}, either $j$ is a minimal generator or $\mathfrak{d}(j) \ge 2$. In the latter case, by Theorem~\ref{thm:betti-minimals:omega}, there is $\alpha \in \operatorname{Minimals}_{\le_S}\mathcal{E}(S)$ with $\alpha \le_S j$, and that $e_{\alpha} > 0$. However, this implies that 
  $$\min \{d \in \mathbb{N}: e_{d} > 0\} \le \alpha < j < \min \{d \in \mathbb{N}: e_{d} > 0\},$$ a contradiction. We conclude that $j$ must be a minimal generator of $S$.
\end{proof}

\subsection{Open questions}

Regarding cyclotomic exponent sequences of arbitrary numerical semigroups, it would be interesting to study the values of these sequences at those Betti elements that are not in $\operatorname{U}(\operatorname{Betti}(S))$, where our current techniques fail to yield any result. 

Coming back to cyclotomic numerical semigroups, by Definition \ref{def:CNS} a numerical semigroup is cyclotomic if and only if its cyclotomic exponent sequence has finitely many non-zero terms. By  
Proposition \ref{prop:cyclotomic} this is equivalent
with $\operatorname{P}_S$ being a product of cyclotomic polynomials.
\begin{question}
Is there a weaker condition than 
the cyclotomic 
exponent sequence having finite support that would 
ensure that $\operatorname{P}_S$ is a product of cyclotomic polynomials?
\end{question}
A possible way to weaken the condition would be, for instance, to require that the 
exponent sequence has infinitely many zeros.

We point out that Conjecture~\ref{con:cyclo-ci} remains open, and it seems likely that further tools are needed in order to tackle it. One could start by showing that if $S$ is a numerical semigroup such that $|\mathcal{E}(S)| \le  2$, then $S$ is a complete intersection. Here we have managed to address the case  $|\mathcal{E}(S)| = 1$ in Theorem~\ref{thm:ces:char}, but our techniques are not enough to analyze the case when $|\mathcal{E}(S)| = 2$ and the two elements in $\mathcal{E}(S)$ are incomparable with respect to $\le_S$. As seen in this section, Conjectures~\ref{con:msg} and~\ref{con:betti} are equivalent with Conjecture~\ref{con:cyclo-ci}. It is thus well possible that at least one of the two is considerably easier than Conjecture~\ref{con:cyclo-ci}, and thus they warrant individual investigation.

\section*{Acknowledgments} 
Part of the work on this paper was done during an internship in the Fall of 2016 carried 
out by the third author at the Max Planck Institute for 
Mathematics in Bonn and 
during a one-week visit in April 2019. He would like to thank the fourth author for the invitation and the institute staff for their hospitality and support. The project was completed during a stay of the first author at the same institute. Substantial progress on this paper was made in February 2017, when the first and the fourth author were invited by the second and third author for one week to the University of Granada. They are grateful for the hospitality, for the inspiring and cheerful atmosphere and, last but not least, for the excellent tapas and wine!


\end{document}